\documentclass[12pt]{elsarticle}

\usepackage[utf8]{inputenc}
\usepackage[margin=1in]{geometry}
\usepackage{amsmath}
\usepackage{amssymb}
\usepackage{amsthm}
\usepackage{thmtools}
\usepackage{hyperref, cleveref}
\usepackage{bm}
\usepackage{bbm}
\usepackage{parskip}
\usepackage{xcolor}
\usepackage{mathrsfs}
\usepackage{comment}
\usepackage{tikz}
\usepackage{enumitem}
\usepackage{overpic}
\usetikzlibrary{decorations.pathreplacing,intersections,arrows.meta}

\declaretheorem[numberwithin=section]{theorem}
\declaretheorem[numberwithin=section, sibling=theorem]{lemma, corollary, proposition}
\declaretheorem[numbered=no]{definition}
\declaretheorem[numberwithin=section, sibling=theorem, style=remark]{example}

\declaretheorem[name=Main Result]{mainresult}
\crefname{mainresult}{main result}{main results}
\Crefname{mainresult}{Main Result}{Main Results}

\newtheorem*{corollary*}{Corollary}
\newtheorem*{proposition*}{Proposition}
\newtheorem*{lemma*}{Lemma}
\newtheorem*{theorem*}{Theorem}

\newcommand{\R}{\mathbb{R}}

\newcommand{\N}{\mathbb{N}}
\newcommand{\Q}{\mathbb{Q}}
\newcommand{\F}{\mathbb{F}}

\newcommand{\set}[1]{{\left\{#1\right\}}}
\newcommand{\paren}[1]{{\left(#1\right)}}
\newcommand{\parensize}[2]{{#1(#2#1)}}
\newcommand{\bracket}[1]{{\left[#1\right]}}
\newcommand{\floor}[1]{{\left\lfloor#1\right\rfloor}}
\newcommand{\ceil}[1]{{\left\lceil#1\right\rceil}}
\newcommand{\norm}[1]{\left\lVert#1\right\rVert}
\newcommand{\normsize}[2]{#1\lVert#2#1\rVert}

\newcommand{\abs}[1]{\left|#1\right|}
\newcommand{\abssize}[2]{#1|#2#1|}

\renewcommand{\bar}[1]{\overline{#1}}

\newcommand{\pmat}[1]{\begin{pmatrix}#1\end{pmatrix}}

\let\mod\undefined
\DeclareMathOperator{\mod}{mod}
\let\Re\undefined
\DeclareMathOperator{\Re}{Re}

\DeclareMathOperator{\Lip}{Lip}

\renewcommand{\epsilon}{\varepsilon}

\renewcommand{\theequation}{\arabic{section}.\arabic{equation}}
\usepackage{chngcntr}
\counterwithin{equation}{section}
\newcommand{\AnchorRight}{\hspace{1em}&\hspace{-1em}}

\crefname{equation}{}{} 

\creflabelformat{equation}{(#2#1#3)}

\newcommand{\eqlabel}[1]{\stepcounter{equation}\tag{\theequation}\label{#1}} 

\newcommand{\linfuns}{\mathrm{LI}}

\newcommand{\todo}[1]{\textcolor{red}{#1}}
\newcommand{\note}[1]{}

\makeatletter
\def\th@plain{%
  \thm@notefont{}
  \itshape 
}
\def\th@definition{%
  \thm@notefont{}
  \normalfont 
}
\def\th@remark{%
  \thm@notefont{}
  \thm@headfont{\itshape}
  \normalfont 
}
\makeatother

\begin{document}

\begin{frontmatter}



\title{Delay-Independent Stability of Nonlinear Delay Differential Equations via Isospectral Reduction}

\author[1]{Quinlan Leishman\corref{cor1}}
\cortext[cor1]{Corresponding author}
\ead{quinlan.leishman@math.byu.edu}
\author[1]{Benjamin Webb} 
\ead{bwebb@math.byu.edu}

\affiliation[1]{organization={Department of Mathematics},
            addressline={Brigham Young University}, 
            city={Provo},
            postcode={84602}, 
            state={Utah},
            country={United States}}

\begin{abstract}
Time delays arise naturally in a wide range of natural and technological systems, yet their influence on the stability remains a challenge to characterize, particularly for nonlinear systems. In this paper, we {develop} a stability framework that yields a delay-independent criterion for global exponential stability in a broad class of nonlinear, nonautonomous delay differential equations. Our approach is based on a novel method that associates the delayed system with a sequence of finite-dimensional matrices of increasing size, which are analyzed using the graph-theoretic technique of isospectral reduction. In contrast to most existing results for nonlinear delay differential equations, which rely on Lyapunov-based methods, our framework provides a general and computationally efficient alternative. As an application, we apply this criterion to analyze consistency in delayed reservoir computing systems, illustrating how the proposed approach can be used to assess stability properties relevant to prediction tasks.
\end{abstract}


\end{frontmatter}



\section{Introduction}

Delay differential equations (DDEs) arise naturally in the modeling of many natural and technological systems, where non-negligible processing or transmission times introduce dependence on past states. Such delays occur in a wide range of settings, including controlled systems \cite{LogemannTownley1996}, transportation networks \cite{TREIBER200671}, economic markets \cite{BelairMackey1989}, supply chains \cite{RiddallsBennett2002}, population dynamics \cite{Zhao2017}, machining processes \cite{StoneCampbell2004}, biological systems \cite{VIELLE1998105,Rihan2021biologyapplications}, and, more recently, machine learning architectures \cite{tavakoli2024delayedrc}. In many of these applications, the underlying delays are not constant but instead vary over time due to changing operating conditions, environmental influences, or adaptive system behavior.

In this paper, we consider delay differential equations of the form
\begin{equation}\label{eqn:dde}
x'(t)=f\bigl(t, x(t),x(t-h_1(t)),\ldots,x(t-h_r(t))\bigr), 
\quad x(t)=\phi(t)\ \text{for } t\leq 0,
\end{equation}
where $x(t)\in \mathbb{R}^n$ or $\in\mathbb{C}^n$ and the functions $h_i(t)$, $i=1,\ldots,r$, represent delays that are potentially time-dependent. The presence of these delays can significantly affect system dynamics and is often a source of instability and degraded performance.

A central problem in the analysis of \eqref{eqn:dde} is the stability of its \emph{equilibrium points} where an equilibrium point $x_0\in\mathbb{R}^n$ satisfies $f(t,x_0,\ldots,x_0)=0$ for all $t$. In this work, we focus on \emph{global exponential stability}: an equilibrium $x_0$ is globally exponentially stable if there exist constants $C,\gamma>0$ such that every solution $x(t)$ of  satisfies
\[
\|x(t)-x_0\|\leq C e^{-\gamma t}\sup_{s\leq 0}\|x(s)-x_0\|, 
\quad t\geq 0.
\]

Our primary interest is in \emph{delay-independent stability}, that is, conditions under which stability holds for a prescribed class of delay functions $h_1(t),\ldots,h_r(t)$. This notion is of particular interest in applications, as it ensures stability without requiring detailed knowledge of the delays themselves. Systems satisfying delay-independent stability are therefore inherently resilient to variations in delay arising from internal dynamics or external perturbations.
 
Much of the existing literature on the stability of delay differential equations relies on Lyapunov-based methods, including Lyapunov–Krasovskii functionals and Razumikhin techniques, to obtain sufficient stability conditions \cite{mahto2020,alexandrova2023,alexandrova2018,diab2024,chen2018} While these methods are quite powerful and widely applicable, they typically require the construction of system-specific functionals and can be difficult to devise for general nonlinear, nonautonomous systems. Moreover, the resulting conditions often depend explicitly on the form or magnitude of the delays.

In contrast, the approach developed in this work does not rely on Lyapunov constructions. Instead, we associate the delay differential equation with a sequence of finite-dimensional systems and analyze their structure using graph-theoretic techniques, yielding delay-independent conditions that are both general and computationally tractable.

The main contribution of this paper is a sufficient condition guaranteeing global exponential stability of \eqref{eqn:dde} that is independent of the delays, under general assumptions on the nonlinear function $f$ and for all continuous, bounded delay functions (see Main Result~\ref{mr1}). In addition, we derive an explicit bound on the exponential convergence rate of the system (see Main Result~\ref{mr2}). Finally, we show that when this condition holds, the system exhibits a strong form of asymptotic synchronization: in the absence of an equilibrium, any two trajectories converge exponentially to the another, implying asymptotic independence to initial conditions (see Main Result~\ref{mr3}).

Our results rely on the following matrices. For the function \(f = f(t, x, y_1, \ldots, y_r)\), let \(D_x f\) and \(D_{y_1} f, \ldots, D_{y_r} f\) denote the matrices of partial derivatives of \(f\) with respect to each argument. We define the \emph{stability matrices}
\begin{equation}\label{eqn:stabmat}
M_0 = \sup_{t \ge 0,\, x \in \mathcal{F}^n} \mathrm{abs}^*(D_x f),
\qquad
M_i = \sup_{t \ge 0,\, x \in \mathcal{F}^n} \abs{D_{y_i} f},
\quad 1 \le i \le r,
\end{equation}
where the suprema are taken entrywise. For a matrix \(A \in \mathbb{C}^{n \times n}\), let \(\abs{A}\) denote the entrywise absolute value of \(A\), and let \(\mathrm{abs}^*(A)\) denote the matrix whose diagonal entries are \(\Re(A_{ii})\) and whose off-diagonal entries are \(\abs{A_{ij}}\). Finally, for a square matrix \(A\), let
\[
\alpha(A) = \sup \Re \sigma(A)
\]
denote the \emph{spectral abscissa} of \(A\), i.e., the largest real part among its eigenvalues.

\begin{mainresult}\textbf{(Delay Independent Stability)}\label{mr1}
Suppose $f$ has the fixed point $x_0\in\mathbb{R}^n$. If the stability matrices in Equation \eqref{eqn:stabmat} satisfy the condition
\begin{equation}\label{eqn:stabcond}
        \alpha\paren{
            M_0+\cdots+M_r
        }<0,
\end{equation}
then $x_0$ is  globally exponentially stable for any set of continuous bounded time delays $h_i(t)$ for $i=1,\dots r$.
\end{mainresult}   

\begin{mainresult}\textbf{(Rate of Convergence)}\label{mr2}
Suppose $f$ has the fixed point $x_0\in\mathbb{R}^n$ and that condition \eqref{eqn:stabcond} holds. If the time delay $h_i(t)\leq T$ for each $i=1,\dots,r$ then 
    \[\norm{x(t)-x_0}\leq Ce^{-\gamma t}\sup_{s\leq 0}\norm{\phi(s)-x_0}\] 
for any solution $x(t)$ of Equation \eqref{eqn:dde}, for any $0<\gamma<\gamma_0$, where $\gamma_0$ is given by
    \begin{equation}\label{eqn:growth-rate2}
    \gamma_0=\inf\set{\gamma\in \R:\,\det\paren{\gamma I+\,M_0+e^{\gamma T}\sum_{i=1}^rM_i}=0
    }>0.
    \end{equation}
\end{mainresult}

\begin{mainresult}\textbf{(Independence to Initial Conditions)}\label{mr3}
Suppose $f$ has continuous, bounded delays \(h_i(t)\) for \(i = 1, \ldots, r\) and satisfies condition~\eqref{eqn:stabmat}. Then system~\eqref{eqn:dde} exhibits \emph{exponential independence of initial conditions}: for any solutions \(x_1(t)\) and \(x_2(t)\) with initial conditions \(\phi_1\) and \(\phi_2\),
\[
\|x_1(t) - x_2(t)\| \le C e^{-\gamma t} \sup_{s \le 0} \|\phi_1(s) - \phi_2(s)\|
\]
for any $0<\gamma<\gamma_0$, where $\gamma_0$ is given by~\eqref{eqn:growth-rate2},
and hence $\|x_1(t) - x_2(t)\|\rightarrow 0$ exponentially as $t\rightarrow \infty$.
\end{mainresult}


These results offer two principal advances over existing work. First, they establish global stability for a broad class of nonlinear systems. In contrast, most previous results  on delay-independent stability are restricted to linear systems and extend to nonlinear dynamics only locally via linearization around equilibria. Second, our results guarantee exponential stability, whereas prior studies have primarily focused on asymptotic stability. For example, Theorem 3.1 in~\cite{sun2012delayindep} treats asymptotic stability for switched linear delay systems. In addition, our approach yields an explicitly computable upper bound on the exponential rate of convergence.

Our approach for establishing exponential stability for systems given by Equation~\eqref{eqn:dde} builds on a sequence of results for discrete-time delayed dynamical systems~\cite{Webb2013, Reber_2020, Carter_2022}. In~\cite{Webb2013}, Bunimovich and Webb introduced the notion of \emph{intrinsic stability} for nonlinear discrete-time systems with delays, showing that intrinsic stability implies global exponential stability independent of the choice of delays. Their analysis relies on a graph-theoretic technique known as \emph{isospectral reduction}. These methods were subsequently extended to systems with time-varying delays in~\cite{Reber_2020} and to stochastically varying delays in~\cite{Carter_2022}. In the present work, these techniques are adapted and extended for the first time to the continuous-time setting, which allow us to establish Main Results~\ref{mr1}–\ref{mr3}. We therefore refer to the stability described in these results as \emph{intrinsic stability}. 

We further show that the proposed framework extends beyond global stability. 
In particular, the same techniques yield local stability results (see \Cref{thm:mr-corollary3}) and, in certain settings, can be used to establish the existence of fixed points or periodic trajectories (see \Cref{thm:mr-corollary1,thm:mr-corollary2}). 
We also generalize our main results to more complex delays in Section~\ref{sec:discretization}, allowing, for example, delays with certain types of discontinuities.

As an application of the theory, Section~\ref{sec:application} considers reservoir computers, a class of machine learning methods designed to learn the underlying dynamics of a system from time-series data. A key property associated with strong performance in this setting is \emph{consistency}~\cite{consistency2019}. Recent work has shown that incorporating time delays into reservoir dynamics can further enhance performance~\cite{tavakoli2024delayedrc}. Building on these observations, we show that, under mild assumptions, delayed reservoirs exhibit consistency even when the delays are allowed to vary in time (see Theorem~\ref{thm:consistency-delayed-2}).

The paper is organized as follows. In Section~\ref{sec:examples}, we present several illustrative examples and applications of Main Results~\ref{mr1}--\ref{mr3}, including a case demonstrating the necessity of the bounded-delay assumption. 
Sections~\ref{sec:comparison}--\ref{sec:gsr-bound} are devoted to the proofs of Main Results~\ref{mr1}--\ref{mr3}. 
In Section~\ref{sec:comparison}, we relate the behavior of the DDE~\eqref{eqn:dde} to that of its global linearization (see \Cref{thm:bound-by-linear}). 
Section~\ref{sec:discretization} analyzes this linearization for a subclass of time delays and extends the results to more general types of delays. In Section~\ref{sec:gsr-bound}, we study the asymptotic behavior of a sequence of matrices associated with the global linearization, completing the proofs of the main results. In Section~\ref{sec:corollaries}, we apply these results to establish criteria for local stability and the existence of fixed points and limit cycles. Section~\ref{sec:application} presents an application to consistency in delayed reservoir computing. We conclude by outlining several directions for future work.

\section{Examples and Applications}\label{sec:examples}

Time delays typically slow convergence and may destabilize otherwise stable systems. We illustrate these effects with two examples and demonstrate how our results apply.

\begin{figure}
\center
\begin{overpic}[width=0.9\textwidth]{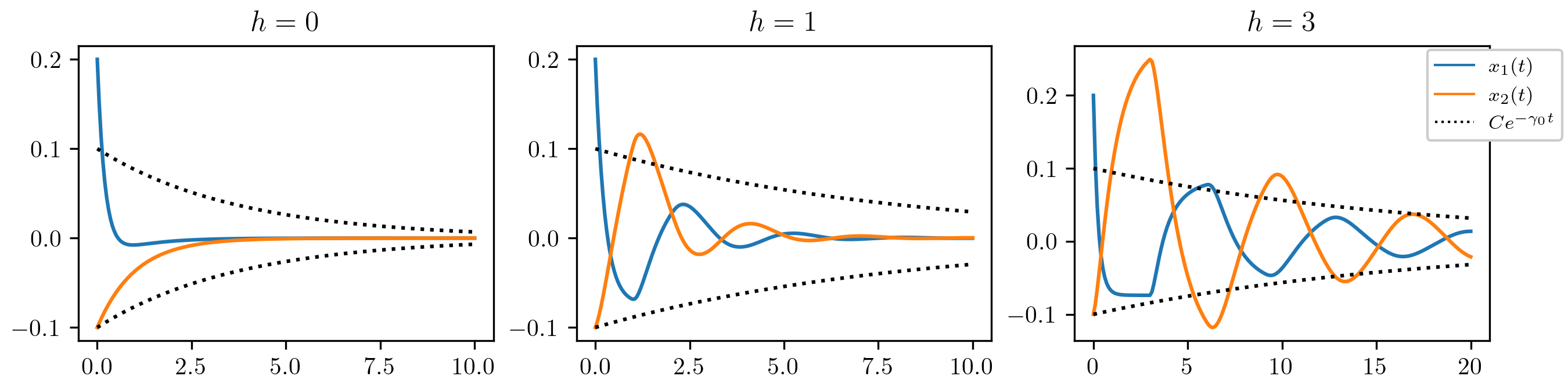}
\put(16,-3){(a)}
\put(48,-3){(b)}
\put(80,-3){(c)}
\end{overpic}
\vspace{0.25in}
\\
\begin{overpic}[width=0.8\textwidth]{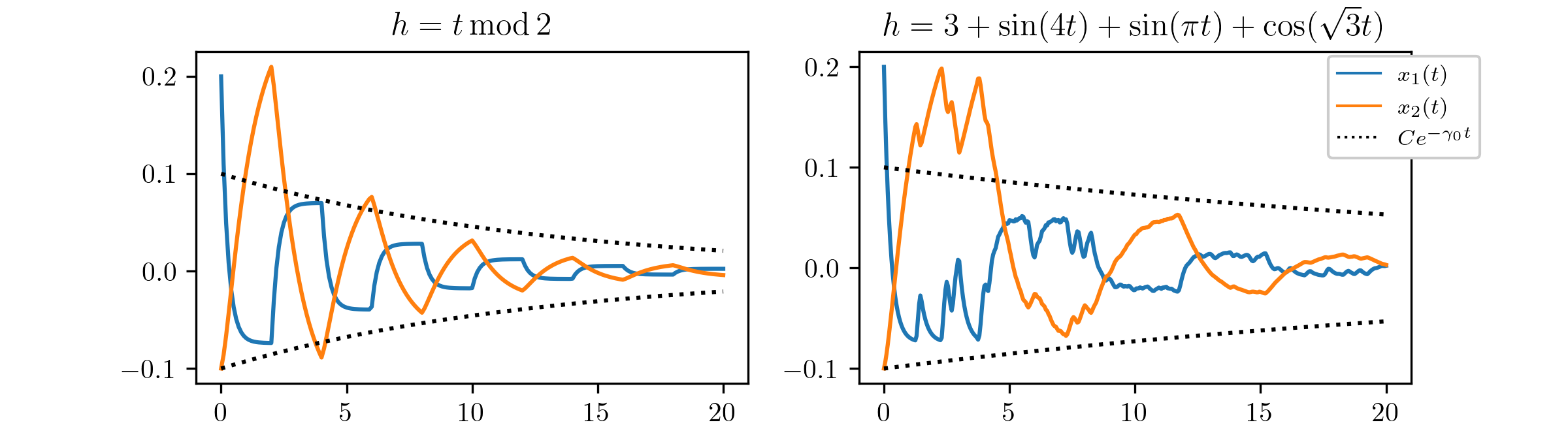}
\put(28,-3){(d)}
\put(71,-3){(e)}
\end{overpic}
\vspace{0.1in}
\caption[Intrinsic stability example with constant delays]{Trajectories of the intrinsically stable system \eqref{eqn:is-system} for several choices of time delay $h$ are shown, together with the asymptotic bound predicted by \Cref{mr2}. Panels (a)–(c) show constant delays: (a) $h=0$ (no delay), (b) $h=1$, and (c) $h=3$. Panel (d) shows a periodic delay $h(t)=t \bmod 2$, while panel (e) shows an aperiodic delay $h(t)=3+\sin(4t)+\sin(\pi t)+\cos(\sqrt{3}\,t)$. In all cases, the fixed point $(0,0)\in\mathbb{R}^2$ remains exponentially stable.
}\label{fig:is-system}
\end{figure}

\begin{example}[\textbf{Intrinsic Stability}]\label{ex:is-system}
Consider the DDE
\begin{equation}\label{eqn:is-system}
x'(t)=Ax(t)+\sin(Bx(t-h)),
\quad A=\pmat{-4 & 0 \\ -1 & -1},\quad B=\pmat{-1 & 1 \\ 1 & 0},
\end{equation}
which has a fixed point at $(0,0)\in\mathbb{R}^2$. Figure~\ref{fig:is-system}(a)–(c) shows that increasing constant delays only slow convergence while preserving stability. Figures~\ref{fig:is-system}(d)–(e) illustrate the effects of periodic and aperiodic time-varying delays. Although the delay in (d) is discontinuous, this case is covered by the extension discussed at the end of Section~\ref{sec:discretization}.

Setting $f(t,x,y)=Ax+\sin(By)$, we compute
\[
D_xf(t,x,y)=A,\quad
D_yf(t,x,y)=\cos(By)^T\odot B,
\]
where $\odot$ denotes the componentwise product. Since
\[
M_0+M_1=\mathrm{abs}^*(A)+\abs{B}=\pmat{-3 & 1 \\ 2 & -1},
\quad \alpha(M_0+M_1)=-2+\sqrt{3}<0,
\]
the system is intrinsically stable and {therefore} globally exponentially stable by Main Result~\ref{mr1}. {Importantly,} this verification is no more complex than analyzing the stability of the undelayed system.
\end{example}

Our second example is adapted from the example in \cite[Section 4]{Louisell2001}.

\begin{example}[\textbf{Non-Intrinsic Stability}]\label{ex:nis-system}
Consider the nonlinear DDE
\begin{equation}\label{eqn:nis-system}
x'(t)=Ax(t)+\sin(Bx(t-h)),
\quad A=\pmat{-1 & 0 \\ 0 & -1},\quad
B=\pmat{-5/4 & 1/4 \\ 1/4 & -5/4},
\end{equation}
which has an equilibrium at $(0,0)\in\mathbb{R}^2$. Standard methods {guarantee} stability only for delays $h\in[0,2.0576)$ (see \cite{Louisell2001,GuStabilityBook}), {and do not provide delay-independent guarantees.} 

Figure~\ref{fig:nis-system}(a)–(c) illustrates that increasing constant delays beyond this regime {(from $h=1$ to $h=3$) leads to loss of stability. Figure~\ref{fig:nis-system}(d) shows a similar destabilization under time-varying delays $h(t)=t \bmod 2$, despite stability holding for each fixed value $h\leq 2$.}

{This behavior can be understood in terms of the intrinsic stability condition:} the associated stability matrices satisfy
\[
M_0+M_1=\mathrm{abs}^*(A)+|B|=\pmat{1/4 & 1/4 \\ 1/4 & 1/4},
\]
for which $\alpha(M_0+M_1)=1/2>0$. Hence condition~\eqref{eqn:stabcond} fails, {and the system is not intrinsically stable.}
\end{example}

\begin{figure}
\center
\begin{overpic}[width=0.9\textwidth]{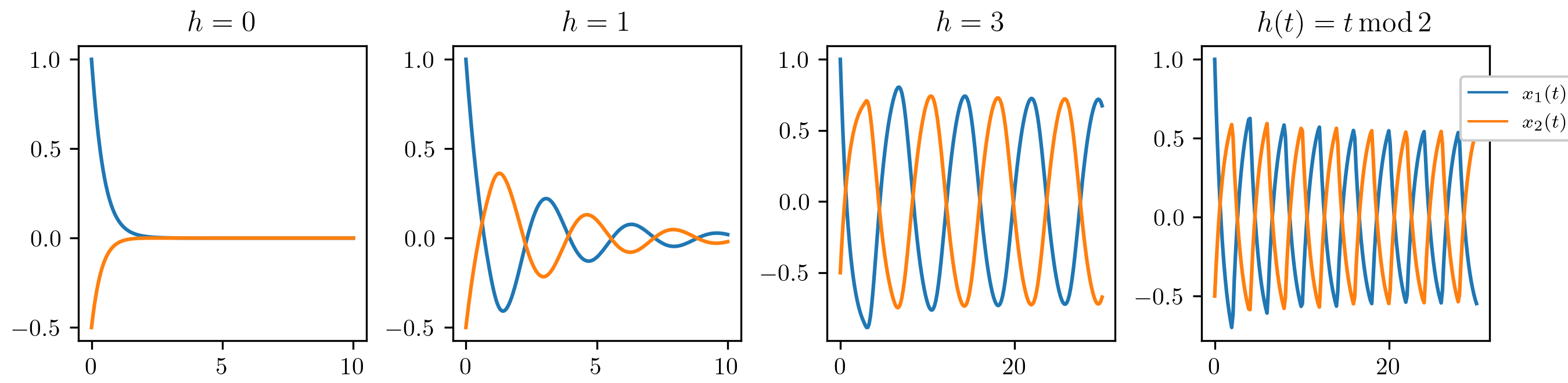}
\put(12,-3){(a)}
\put(36,-3){(b)}
\put(60,-3){(c)}
\put(84,-3){(d)}
\end{overpic}
\vspace{0.1in}
\caption[Intrinsic stability non-example]{Trajectories of the non-intrinsically stable system \eqref{eqn:nis-system} {are shown for constant and time-varying delays} $h$. Frames {(a)–(c) show constant delays: (a) $h=0$, (b) $h=1$, and (c) $h=3$.} For small delays the system remains stable, while larger delays lead to destabilization. Frame (d) shows a periodic time-varying delay $h(t)=t \bmod 2$. Although the system is stable for every fixed $h\leq 2$, this time-varying delay also {induces instability}.
}\label{fig:nis-system}
\end{figure}

Example~\ref{ex:nis-system} highlights a setting {where} classical delay-dependent criteria are {insufficient} \textcolor{black}{to describe the behavior under varying delays}. {In contrast,} the intrinsic stability framework provides a delay-independent characterization that applies to both constant and time-varying delays.

Our third example illustrates that for systems satisfying Equation \eqref{eqn:stabcond}, exponential stability may fail to hold if the time delays are unbounded.
\begin{example}[\textbf{Unbounded Time-Delays}]\label{ex:unbounded}
   Consider the linear DDE
\begin{equation}
x'(t)=-x(t)+\frac12 x(t-h(t)),\quad t\geq 1,
\end{equation}
with unbounded delay
\[
h(t)=t-\Big(\sqrt{t}+\log\!\Big(\frac{2\sqrt{t}}{2\sqrt{t}-1}\Big)-\log(2)\Big)^2.
\]

The system satisfies $M_0=[-1]$ and $M_1=[1/2]$, {so condition~\eqref{eqn:stabcond} holds. However,} it does not satisfy the hypotheses of Main Result~\ref{mr1} since $\sup_{t\ge 0} h(t)=\infty$. \textcolor{black}{Moreover,} the equilibrium $0$ is not exponentially stable.
{Indeed}, $x(t)=e^{-\sqrt{t}}$ is a solution, since
\begin{align*}
-x(t)+\frac12 x(t-h(t))
&= -e^{-\sqrt{t}}
+ \frac12 \exp\!\Big(-\sqrt{t}-\log\!\Big(\frac{2\sqrt{t}}{2\sqrt{t}-1}\Big)+\log 2\Big) \\
&= -e^{-\sqrt{t}} + e^{-\sqrt{t}}\Big(1-\frac{1}{2\sqrt{t}}\Big)
= -\frac{1}{2\sqrt{t}}e^{-\sqrt{t}} = x'(t).
\end{align*}

Since this solution decays {only} sub-exponentially, the equilibrium is not exponentially stable. (We {note} that earlier results of Sun~\cite{sun2012delayindep} imply that the equilibrium is globally asymptotically stable.)
\end{example}

\section{Global Linearization of Nonlinear Systems}\label{sec:comparison} 
In this section, we show that the growth rate of a nonlinear system can be bounded by that of an associated linear system, which we call the \emph{global linearization} (see Equation~\eqref{eqn:dde-linearized} and Theorem~\ref{thm:bound-by-linear}). This linear system serves as a proxy for determining \emph{intrinsic stability}, namely, whether the equilibrium is globally exponentially stable for all continuous bounded time delays. The resulting comparison principle is the main tool in the proof of Main Results~\ref{mr1}--\ref{mr3}.

\begin{definition}[Global linearization]
   To the nonlinear equation \eqref{eqn:dde} we associate the following linear DDE:
\begin{equation}\label{eqn:dde-linearized}
        y'(t)=M_0y(t)+\sum_{i=1}^r M_iy(t-h_i(t)),
    \end{equation}
where the $M_i$ are the stability matrices given by Equation \eqref{eqn:stabmat}.
We refer to Equation \eqref{eqn:dde-linearized} the \emph{global linearization} of Equation \eqref{eqn:dde}. 
\end{definition}

We begin our analysis of Equation \eqref{eqn:stabmat} by introducing some notation. First, the notion of a flow map will be useful. For any fixed $t_1$, any initial condition $x(t)=\phi(t-t_1)$ for $t\in(-\infty,t_1]$ defines a unique solution $x$ of Equation \eqref{eqn:dde} for $t\geq t_1$.
This induces a (non-autonomous) flow on the space $C((-\infty,0];\R^d)$, with flow map $S_{[t_1,t_2]}^h$ defined by
\[
S_{[t_1,t_2]}^h[\phi](s)= x(t_2+s)
\]
for $s\in(-\infty,0]$.
This corresponds to advancing an initial condition on $(-\infty,t_1]$ to a solution on the interval $(-\infty,t_2]$.
We will denote the analogous flow map for the global linearization \eqref{eqn:dde-linearized} as $R_{[t_1,t_2]}^h$.
For brevity, we will also use the shortened notation $S^h_t:=S^h_{[0,t]}$ and $R^h_t:=R^h_{[0,t]}$.
Note that as Equation \eqref{eqn:dde-linearized} is linear, $R_t^h$ is in fact a bounded linear operator on $C((-\infty,0])$. 
Since $M_0$ is a Metzler matrix and $M_1,\ldots,M_r$ are nonnegative, it is a standard result that solutions to Equation \eqref{eqn:dde-linearized} with nonnegative initial conditions will be nonnegative for all time, or equivalently that $R_t^h$ is a nonnegative operator; see e.g. \cite[Lemma 3]{liu2011delayindep}.

Occasionally it will be useful to regard these flow maps as mapping between certain other spaces. 
In particular, observe that both of these flow maps also map $C^{0,1}((-\infty,0])$ into itself. If the time delays are bounded as $h_i(t)\leq T$, then these may moreover be regarded as maps on the spaces $C([-T,0])$ or on $C^{0,1}([-T,0])$.

Second, we introduce a convention for notating multiple time delays.
For a function $x(t):\R\to \F^d$ and time delays $h$, let $\Delta_h x:\R\to \F^{d\times r}$ be given by $\Delta_h x(t)=(x(t-h_1(t)),\ldots,x(t-h_r(t)))$.
With this notation, we may write the DDE \eqref{eqn:dde} in the more compact form
\begin{equation*}\tag{\ref{eqn:dde}$'$}\label{eqn:dde-short}
x'(t)=f(t,x(t),\Delta_hx(t)).
\end{equation*}
For a fixed vector of constant time delays $v\in \R^n$ and a function $\phi$, we will also use this notation to mean $\Delta_{v}\phi=(\phi(-v_1),\ldots,\phi(-v_r))$.

Finally, the standard theory of continuous dependence of solutions on initial conditions and parameters for DDEs parallels that of ODEs. Since we will use these results throughout the paper, we collect them here for convenience.

\begin{proposition}\label{thm:continuity}
Let $T$ be such that $h_i(t)\leq T$ ($T=\infty$ if $h$ is unbounded), and denote $C^0=C([-T,0])$ and $C^{0,1}=C^{0,1}([-T,0])$
The following hold for some constants $C,L>0$:
\begin{enumerate}[label=\textup{(\roman*)}]
\item \note{Uniform continuity} \note{$t$-boundedness doesn't let us simplify to 1-arg versions as need continuity}
	$\displaystyle \norm{S_t^h[\phi_1]-S_t^h[\phi_2]}_{C^0}\leq e^{Lt}\norm{\phi_1-\phi_2}_{C^0}$
\item 
	$\displaystyle \Lip(S_t^h[\phi_1]-S_t^h[\phi_2])\leq \begin{cases}
	\Lip(\phi_1-\phi_2) + Ce^{Lt}\norm{\phi_1-\phi_2}_{C^0} & \text{$t < T$ or $T=\infty$}
	\\
	Ce^{Lt}\norm{\phi_1-\phi_2}_{C^0} & t \geq T
	\end{cases}$
\item
	$\displaystyle
	\norm{S_t^h[\phi_1]-S_t^h[\phi_2]}_{C^{0,1}}\leq Ce^{Lt}\norm{\phi_1-\phi_2}_{C^{0,1}}
	$
\item  $
	\norm{S^{g}_t[\phi]-S^{h}_t[\phi]}_{C^{0,1}}
	\leq 
	C e^{Lt}(1+\norm{\phi}_{C^{0,1}})\norm{g-h}_{L^\infty([0,t])}
	$
\end{enumerate}
\end{proposition}
The proof of \Cref{thm:continuity} is in Appendix A.
The analogous statements also hold for $R_t^h$.
Note that (iv) in particular implies $R_t^h$ varies continuously in the $C^{0,1}$-operator norm topology as $h$ varies in the $L^\infty$-norm.

We now turn to the main results of this section, whose goal is to establish a connection between the nonlinear system \eqref{eqn:dde} and its global linearization \eqref{eqn:dde-linearized}.

\begin{theorem}\label{thm:bound-by-linear}
Suppose that $f$ is $C^1$ with bounded first derivatives, and that the vector of time delays $h(t)$ is a regulated function.
Then, for any solutions $x_1(t),x_2(t)$ of the nonlinear system \eqref{eqn:dde}, if $y(t)$ is the solution to the global linearization \eqref{eqn:dde-linearized} with initial condition $y(t)=\abs{x_1(t)-x_2(t)}$ for $t\leq 0$, then for all $t\in \R$,
\begin{equation}\label{eqn:bound-by-linear1}
    \abs{x_1(t)-x_2(t)} \leq y(t).
\end{equation}
\end{theorem}

This theorem {provides a bound on} solutions of Equation~\eqref{eqn:dde} {in terms of} solutions of the global linearization (see Equation \eqref{eqn:dde-linearized}). In particular, it allows us to show that stability of the global linearization implies Main Results~\ref{mr1}--\ref{mr3} (see Section~\ref{sec:discretization}).

Recall that a function is \emph{regulated} if it is in the $L^\infty$-closure of step functions.
Stated in terms of the flow maps $S_t^h$ and $R_t^h$, Equation \eqref{eqn:bound-by-linear1} may be written as
\begin{equation}\label{eqn:bound-by-linear2}
    \abs{S^h_t[\phi_1]-S^h_t[\phi_2]}\leq R^h_t[\abs{\phi_1-\phi_2}]
\end{equation}
for any initial conditions $\phi_1,\phi_2$.
Equation \eqref{eqn:bound-by-linear2} may be interpreted as stating $R_t^h$ acts as a ``linear absolute value'' for the flow map $S_t^h$.

The proof of \Cref{thm:bound-by-linear} relies on two technical lemmas.
In the first we analyze the case that the time delays $h(t)=h$ are constant. To do so, we let $R_{[t_1,t_2]}^{h,\delta}$ denote the flow map for Equation \eqref{eqn:dde-linearized} for $\delta\geq 0$ with $M_0$ replaced by $M_0+\delta I$. In order to assist in the proof, for $\tau>0$ small we furthermore introduce the notation for two discretized flow maps using
\[
S^\tau_t[\phi](-s)=\begin{cases}
\phi(-s+\tau) & -s\leq -\tau
\\
\phi(0)+(\tau-s)f(t, \phi(0),\Delta_h\phi) & -\tau < -s \leq 0
\end{cases}
\]
and
\[
R^{\tau,\delta}[\phi](-s)=\begin{cases}
\phi(-s+\tau) & -s\leq -\tau
\\
\phi(0)+(\tau-s)\paren{
(
{M}_0+\delta I) x(0)
+
\sum_{i=1}^r {M}_i \phi(-h_i)
}
 & -\tau < -s \leq 0
\end{cases}.
\]
These are forward-Euler approximations to $S_{[t,t+\tau]}^h$ and $R_{[t,t+\tau]}^{h,\delta}$, respectively.
Similar to the actual flow maps, the discretized flow maps $S^\tau_t$ and $R^{\tau,\delta}$ both map $C^{0,1}((-\infty,0])$ into itself, and $R^{\tau,\delta}$ is a bounded linear operator.

\begin{lemma}\label{thm:comparison-small-advance-bound}
Suppose the time delays $h\in \R^r$ are constant.
There exist constants $C_1,C_2$ independent of $h$ such that for any Lipschitz continuous $\phi\in C((-\infty,0];\R^d)$ and all sufficiently small $\tau>0$,
\[
\norm{S^h_{[t,t+\tau]}[\phi]-S^\tau_t[\phi]}_{C^{0}}\leq C_1\tau^2\paren{1+\norm{\phi}_{C^0}+\Lip(\phi)},
\]
\[
\norm{R^{h,\delta}_{[t,t+\tau]}[\phi]-R^{\tau,\delta}[\phi]}_{C^{0}}\leq C_2\tau^2\paren{1+\norm{\phi}_{C^0}+\Lip(\phi)}.
\]
\end{lemma}
\begin{proof}
We first compare $S_{[t,t+\tau]}[\phi](-s)$ and $S^\tau_t[\phi](-s)$.
Note that these functions are identical on $[-T,-\tau]$, so it suffices to consider $-s\in (-\tau,0]$.

Let $L=\Lip(f)$, and let $x(t)=S_t^h[\phi](0)$ be the solution to \eqref{eqn:dde} with initial condition $\phi$.
Then, we have
\begin{align*}
&\norm{S^h_{[t,t+\tau]}[\phi](-s)-S^\tau_t[\phi](-s)}
\\&\quad=
\norm{
(\tau - s)f(t, x(0),\Delta_hx(0))
-
\int_0^{\tau - s}
f(t+\xi,x(\xi),\Delta_hx(\xi))\,d\xi
}
\\
&\quad\leq
\int_0^{\tau - s}
\norm{
f(t+\xi,x(\xi),\Delta_hx(\xi))
-f(t, x(0),\Delta_hx(0))
}
\,d\xi
\\
&\quad\leq
\tau\sup_{\xi\in[0,\tau]}
L\paren{
\xi+\norm{x(\xi)-x(0)}+\norm{\Delta_hx(\xi)-\Delta_hx(0)}
}
\\
&\quad\leq
L\tau^2
\paren{
1+(r+1)\Lip_{(-\infty,\tau]}(x)
},
\end{align*}
where $\Lip_{(-\infty,\tau]}(x)$ denotes the Lipschitz constant of $x(t)$ on the interval $(-\infty,\tau]$.
Using \Cref{thm:continuity} implies that
\[
\norm{S_{[t,t+\tau]}[\phi](-s)-S^\tau_t[\phi](-s)}
\leq
\tau^2
L\paren{
1+(r+1)Le^{Lt}\norm{\phi}_{C^0}+(r+1)\Lip(\phi)
}
\]\[
\leq C_1\tau^2\paren{1+\norm{\phi}+\Lip(\phi)}
\]
for a constant $C_1$ independent of $\phi$ and $h$.
The case for $R^{h,\delta}_{[t,t+\tau]}$ is proved identically.
\end{proof}
\begin{lemma}\label{thm:comparison-small-step-approx}
For any $\delta>0$, there is some $\tau_0$ depending on $\delta$ such that if $0\leq \tau < \tau_0$, then
\[
\abs{S^\tau_t[\phi_1]-S^\tau_t[\phi_2]}\leq R^{\tau,\delta}[\abs{\phi_1-\phi_2}]
\]
pointwise for all $t\geq 0$ and any $\phi_1,\phi_2$.
\end{lemma}
\begin{proof}
For brevity, for $u,v\in \R^n$ and $s\in[0,1]$ let $\mathcal I(u,v,s)=(1-s)u+sv$ be the linear path connecting $u$ and $v$.

Since $\abs{S^\tau_t[\phi_1]-S^\tau_t[\phi_2]}(-s)$ and $R^{\tau,\delta}[\abs{\phi_1-\phi_2}](-s)$ are both equal to $\abs{\phi_1-\phi_2}(\tau-s)$ for $-s\leq -\tau$, it suffices to consider $-\tau<-s\leq 0$.
For this case, we proceed by using the integral mean value theorem along a certain path between the points $(\phi_1(0),\Delta_h\phi_1(0))$ and $(\phi_2(0),\Delta_h\phi_2(0))$ in $\F^{d(r+1)}$.
Specifically, we will consider the concatenation $pq_1\cdots q_r$ of the paths given by
\[
p(s)=\parensize\Big{
\mathcal I\parensize\big{\phi_1(0),\phi_2(0),s},
\phi_1(-h_1),\ldots,\phi_1(-h_r)
},
\]\[
q_i(s)=\parensize\Big{
\phi_2(0),\phi_2(-h_1),\ldots,
\mathcal I\parensize\big{\phi_1(-h_i),\phi_2(-h_i),s},
\ldots,\phi_1(-h_r)
}
\]
for $0\leq s\leq 1$, $1\leq i \leq r$.
Hence for $-\tau<s\leq 0$ we may write
\begin{align*}
\hspace{1em}&\hspace{-1em}\abs{S^\tau_t[\phi_1](-s)-S^\tau_t[\phi_2](-s)}
\\&=
\abs{
\phi_1(0)-\phi_2(0)
+(\tau-s)\int_0^1
Df_x\parensize\big{p(\xi)}(\phi_1(0)-\phi_2(0))\,d\xi
}
\\&\quad\quad+
\abs{\sum_{i=1}^r
Df_{y_i}\parensize\big{q_i(\xi)}(\phi_1(-h_i)-\phi_2(-h_i))
\,d\xi
}
\\&\leq
\paren{\int_0^1\abssize\big{(I+ (\tau-s)Df_x(p(\xi))}\,d\xi}\abs{\phi_1(0)-\phi_2(0)}
\\&\quad\quad+(\tau-s)\sum_{i=1}^r M_i\abs{\phi_1(-h_i)-\phi_2(-h_i)}.
\eqlabel{eqn:comparison-small-step-approx-1}
\end{align*}
Noting that $|1+\tau b|=1+\tau \Re(b)+O(\tau^2)$, it is the case that $\abs{(I+ (\tau-s)Df_x}= I+(\tau-s) \mathrm{abs}^*(Df_x)+O(\tau^2)I$ for sufficiently small $\tau$.
Finally, observe that $\mathrm{abs}^*(Df_x)+O(\tau)I\leq M_0+\delta I$ if we choose $\tau$ small enough that the error term is uniformly smaller than $\delta$.
If we apply these inequalities to \eqref{eqn:comparison-small-step-approx-1}, the right-hand-side becomes exactly $R^{\tau,\delta}[\abs{\phi_1-\phi_2}](-s)$, as claimed.
\end{proof}

We are now ready to prove \Cref{thm:bound-by-linear}.
\begin{proof}[Proof of \Cref{thm:bound-by-linear}]
We first consider the case where $\phi_1$ and $\phi_2$ are Lipschitz continuous.
Note that then $|\phi_1-\phi_2|$ is also Lipschitz continuous, so by \Cref{thm:comparison-small-advance-bound} we may write
\[
\abs{S_{[t,t+\tau]}[\phi_1]-S_{[t,t+\tau]}[\phi_2]}
=
\abs{S_t^\tau[\phi_1]-S_t^\tau[\phi_2]} + E^1(\phi_1,\phi_2),
\]\[
R^\delta_{[t,t+\tau]}[\abs{\phi_1-\phi_2}]
=
R^{\tau,\delta}[\abs{\phi_1-\phi_2}]+E^2(\phi_1,\phi_2),
\]
where $E^1(\phi_1,\phi_2)$, $E^2(\phi_1,\phi_2)$ are the error terms from the lemma.
By \Cref{thm:comparison-small-step-approx}, for any $\delta>0$ we have
\[
\abs{S^\tau_t[\phi_1]-S^\tau_t[\phi_2]}\leq R^{\tau,\delta}[\abs{\phi_1-\phi_2}]
\]
for all sufficiently small $\tau$.
Combining these lemmas implies that
\begin{equation*}
\abs{S_{[t,t+\tau]}[\phi_1]-S_{[t,t+\tau]}[\phi_2]}
\leq
R_{[t,t+\tau]}[\abs{\phi_1-\phi_2}]+E(\phi_1,\phi_2)
\end{equation*}
where $E(\phi_1,\phi_2)$ is the error term $E(\phi_1,\phi_2)=E^1(\phi_1,\phi_2)-E^2(\phi_1,\phi_2)$.
By the error bounds in \Cref{thm:comparison-small-advance-bound}, this satisfies \[
\norm{E(\phi_1,\phi_2)}_{C^0}\leq C \tau^2(1+\norm{\phi_1}_{C^{0,1}}+\norm{\phi_2}_{C^{0,1}}).\]
Now, restrict to $\tau$ that divide $t_2-t_1$, set $N_\tau$ to be the integer such that $\tau N_\tau = t_2-t_1$, and set
\[
E_n
=
E(S_{[t_1,t_1+n\tau]}[\phi_1],S_{[t_1,t_1+n\tau]}[\phi_2])
.
\]
We claim first that
\[
\abs{S_{[t_1,t_1+m\tau]}[\phi_1]-S_{[t_1,t_1+m\tau]}[\phi_2]}
\leq
R^{h,\delta}_{[t_1,t_1+m\tau]}[\abs{\phi_1-\phi_2}]
\sum_{n=1}^{m} R^{h,\delta}_{[t_1+n\tau,t_1+m\tau]}[E_{n-1}].
\]
for $m\leq N_\tau$.
This is clear if $m=1$.
For $m>1$, we have inductively that
\begin{align*}
\AnchorRight\abs{S_{[t_1,t_1+m\tau]}[\phi_1]-S_{[t_1,t_1+m\tau]}[\phi_2]}
    \\&=
    \abs{
    S_{[t_1+(m-1)\tau,t_1+m\tau]}\circ
    S_{[t_1,t_1+(m-1)\tau]}[\phi_1]
    -
    S_{[t_1+(m-1)\tau,t_1+m\tau]}\circ
    S_{[t_1,t_1+(m-1)\tau]}[\phi_2]
    }
    \\&\leq
    R^{h,\delta}_{t_1+(m-1)\tau,t_1+m\tau}\bracket{
    S_{[t_1,t_1+(m-1)\tau]}[\phi_1]
    -
    S_{[t_1,t_1+(m-1)\tau]}[\phi_2]
    }+E_{m-1}
    \\&\leq
    R^{h,\delta}_{t_1+(m-1)\tau,t_1+m\tau}\bracket{
        R^{h,\delta}_{[t,t+(m-1)\tau]}[\abs{\phi_1-\phi_2}]
        +
        \sum_{n=1}^{m-1} R^{h,\delta}_{[t_1+n\tau,t_1+(m-1)\tau]}[E_{n-1}]
    }
    +E_{m-1}
    \\&=
    R^{h,\delta}_{[t_1,t_1+m\tau]}[\abs{\phi_1-\phi_2}]+\sum_{n=1}^{m} R^{h,\delta}_{[t_1+n\tau,t_1+m\tau]}[E_{n-1}].
\end{align*}
In particular, for $m=N_\tau$ this states
\begin{equation}\label{eqn:almost-leq-bound}
\abs{S_{[t_1,t_2]}[\phi_1]-S_{[t_1,t_2]}[\phi_2]}
    \leq
    R^{h,\delta}_{[t_1,t_2]}[\abs{\phi_1-\phi_2}]+\sum_{n=1}^{N_\tau} R^{h,\delta}_{[t_1+n\tau,t_2]}[E_n];
\end{equation}
observe that this is essentially Equation \eqref{eqn:bound-by-linear2} with an added error term.
For the error term in \eqref{eqn:almost-leq-bound} and using \Cref{thm:continuity}(iii) we compute
\begin{align*}
\norm{\sum_{n=1}^{N_\tau} R^{h,\delta}_{[t_2-(n-1)\tau,t_2]}[E_n]}_{C^0}
&\leq
\sum_{n=1}^{N_\tau}
C \tau^2 e^{L(t_2-t_1-n\tau)}(1+Ce^{L n\tau}\norm{\phi_1}_{C^{0,1}}+Ce^{L n\tau}\norm{\phi_2}_{C^{0,1}})
\\&\leq
C'
N_\tau
\tau^2
e^{L(t_2-t_1)}
(1+\norm{\phi_1}_{C^{0,1}}+\norm{\phi_2}_{C^{0,1}})
\\&=
\tau
\paren{C'(t_2-t_1)
e^{L(t_2-t_1)}
(1+\norm{\phi_1}_{C^{0,1}}+\norm{\phi_2}_{C^{0,1}})}
\to 0
\end{align*}
as $\tau\to 0$.
Hence taking this limit in Equation \eqref{eqn:almost-leq-bound},
\[
\abs{S_{[t_1,t_2]}[\phi_1]-S_{[t_1,t_2]}[\phi_2]}\leq R^{h,\delta}_{[t_1,t_2]}[\abs{\phi_1-\phi_2}]
\]
for any $\delta>0$.
By standard arguments solutions to \eqref{eqn:dde-linearized} are pointwise continuous in the coefficients of the matrices, so additionally taking $\delta\to 0$ shows the theorem holds for constant time delays.

Now suppose the time delays $h$ are a step function; in particular, let $h(t)=h^i\in \R^r$ for $t\in(t_{i-1},t_i)$ for $t_0<t_1<\ldots<t_{n-1}<t_n$.
In this case, we may write $S^h_{[t_0,t_n]}$ and $R^h_{[t_0,t_n]}$ as a finite composition of the operators with constant time delays.
We proceed by inducting on the number of intervals $n$.
If $n=1$, this is the case considered above. Otherwise, supposing the claim holds for $S^h_{[t_0,t_{n-1}]}$ and $R^h_{[t_0,t_{n-1}]}$, we have
\begin{align*}
\abs{S^h_{[t_0,t_n]}[\phi_1]-S^h_{[t_0,t_n]}[\phi_2]}
&=
\abs{S^{h_n}_{[t_{n-1},t_n]}\circ S^{h}_{[t_0,t_{n-1}]}(\phi_1)-S^{h_n}_{[t_{n-1},t_n]}\circ S^h_{[t_0,t_{n-1}]}(\phi_2)}
\\&\leq
R^{h_n}_{[t_{n-1},t_n]}\paren{
\abs{
S^h_{[t_0,t_{n-1}]}[\phi_1]-S_{[t_0,t_{n-1}]}[\phi_2]
}
}.
\end{align*}
Since $R$ is a nonnegative operator for any choice of time delays, we have by the inductive hypothesis that
\begin{align*}
R^{h^n}_{[t_{n-1},t_n]}\paren{
\abs{
S^h_{[t_0,t_{n-1}]}[\phi_1]-S_{[t_0,t_{n-1}]}[\phi_2]
}
}
&\leq
R^{h^n}_{[t_{n-1},t_n]}\circ R^{h}_{[t_0,t_{n-1}]}\parensize\big{\abs{\phi_1-\phi_2}}
\end{align*}\[=R^h_{[t_0,t_n]}\parensize\big{\abs{\phi_1-\phi_2}}.
\]

By continuity with respect to time delays (\Cref{thm:continuity}(iv)), the result then follows for any time delays $h$ in the $L^\infty$-closure of step functions, for Lipschitz continuous $\phi_1,\phi_2$.

Finally, as the time delays $h$ are locally bounded, we have that $S$ and $R$ are in fact continuous on $C((-\infty,0];\R^d)$ with the topology of uniform convergence on compact sets.
As Lipschitz continuous functions are dense in $C((-\infty,0];\R^d)$ with this topology, the result thus holds for all $\phi_1,\phi_2\in C((-\infty,0])$.
\end{proof}

\section{Global Linearization as a Switched Linear System}\label{sec:discretization}

In this section, we introduce a switched system framework for analyzing time-varying delays in delay differential equations. As illustrated in \Cref{ex:nis-system}, the time dependence of the delays can significantly alter system behavior and may lead to destabilization even when all associated constant-delay versions of the system are stable.

Our objective in this section is to analyze the behavior of the global linearization
\begin{equation}\label{eqn:dde-linearized-v2}
        x'(t)=M_0x(t)+\sum_{i=1}^r M_ix(t-h_i(t))
    \end{equation}
under the influence of varying delays. To analyze this effect, we model time-varying delays as switching between a family of constant-delay systems. This switched-system perspective provides a systematic way of handling changing delays.

We first restrict our attention to a class of time delays that we refer to as \emph{linear increasing delays}, for which the system \eqref{eqn:dde-linearized-v2} can be analyzed as a finite-dimensional linear switched system. Using properties of this system, which will be developed in Section~\ref{sec:gsr-bound}, we establish stability of the global linearization (see \Cref{thm:almost-mr}), which in turn yields Main Results~\ref{mr1}--\ref{mr3}.

Recall that the matrices $M_i$ in Equation \eqref{eqn:dde-linearized-v2} are the stability matrices given by Equation \eqref{eqn:stabmat}.
Throughout the remainder of the proof of Main Results 1--3, we will fix $T>0$ as an upper bound for the time delays, i.e. $h_i(t)\leq T$.
We will also assume that Equation \eqref{eqn:stabcond} holds, that is,
\[
\alpha(M_0+M_1+\ldots+M_r)<0,
\]
where $\alpha$ is the spectral abscissa.
Although Main Results 1-3 are stated in terms of uniformly continuous time delays, the approach we take allows us to prove these results for a strictly larger class of time delays, namely those in the $L^\infty$-closure of the \emph{linear increasing delays} $\mathrm{LI}$.
The linear increasing delays are described as follows:
\begin{definition}[Linear Increasing Delays]
    For an interval $J\subseteq [0,\infty)$ and $\tau>0$, let the \emph{linear increasing delays} with increment $\tau$, denoted $\mathrm{LI}_\tau(J)$, be the set of piecewise-linear functions $f:J\to \R^r$ satisfying the following properties:
    \begin{enumerate}[label=(\roman*)]
        \item $f$ is right-continuous on $\tau \N$ and differentiable on $J\setminus\tau N$,
        \item $f$ maps integer multiples of $\tau$ to integer multiples of $\tau$ that are contained in the interval $[0,T]$, i.e. $f(\tau\N\cap J)\subseteq \tau \N\cap [0,T]$,
        \item For $t\in J\setminus \tau \N$, $f'(t)=(1,\ldots,1)$.
    \end{enumerate}
    Denote the set of all {linear increasing delays} as
    $
    \mathrm{LI}(J)=\bigcup_{\tau >0}\mathrm{LI}_\tau(J),
    $
    and for brevity, label $\mathrm{LI}_\tau=\mathrm{LI}_\tau([0,\infty))$ and $\mathrm{LI}=\mathrm{LI}([0,\infty))$.
\end{definition}

Note that conditions (i) and (iii) {imply that} a function $f \in \mathrm{LI}_\tau(J)$ is uniquely {determined by its} values $\{f(n\tau)\}_{n \in \mathbb{N}}$. As we will see shortly, the system \eqref{eqn:dde-linearized-v2} {can} be analyzed explicitly for time delays $h \in \mathrm{LI}$.

We note that the assumption $h \in \overline{\mathrm{LI}}$ includes, {in particular,} the case where $h$ is uniformly continuous. {This is formalized in} the following proposition, {and corresponds to the setting} considered in Main Results~\ref{mr1}--\ref{mr3}.

\begin{proposition}\label{thm:uc-satisfies-assumption}
    Any bounded uniformly continuous $h$ is in the $L^\infty$-closure of $\mathrm{LI}$.
\end{proposition}
\begin{proof}
    For any given $\epsilon>0$, since the time delays $h_i$ are uniformly continuous, there exists $\delta>0$ such that $\abs{h_i(x)-h_i(y)}<\epsilon$ if $\abs{x-y}<\delta$ for any $x,y\geq 0$ and for each $i=1,\ldots,r$.
Define the functions $\hat h_i$ by setting
\[
\hat h_i(k\tau)=\tau \floor{\frac{1}{\tau}h_i(k\tau)},
\]
 and continuing linearly with slope 1 on each interval $[k\tau,(k+1)\tau)$; see Figure \ref{fig:discretizaiton-example}.
 The function $\hat h=(\hat h_1,\ldots, \hat h_r)$ is easily verified to be in $\mathrm{LI}$.
If we require $\tau$ to be less than $\min(\epsilon,\delta)$, we can bound 
\[
\normsize{}{h-\hat h}_\infty
\leq \sup_{x\in[0,\tau),k\geq 0}
\Big(
\normsize{}{h(k\tau+x)-h(k\tau)}
+\normsize{}{h(k\tau)-\hat h(k\tau)}\]\[+\normsize{}{\hat h(k\tau)-\hat h(k\tau+x)}
\Big)
<\epsilon+2\tau<3\epsilon
\]
which may be made arbitrarily small.
\end{proof}

\begin{figure}[h]
\begin{center}
\begin{tikzpicture}
	\draw[dotted,gray] (0,-1) -- (8,-1);
	\draw[dotted,gray] (0,0) -- (8,0);
	\draw[dotted,gray] (0,1) -- (8,1);
	\draw[dotted,gray] (0,2) -- (8,2);
	\draw[dotted,gray] (0,-1) -- (0,2);
	\draw[dotted,gray] (1,-1) -- (1,2);
	\draw[dotted,gray] (2,-1) -- (2,2);
	\draw[dotted,gray] (3,-1) -- (3,2);
	\draw[dotted,gray] (4,-1) -- (4,2);
	\draw[dotted,gray] (5,-1) -- (5,2);
	\draw[dotted,gray] (6,-1) -- (6,2);
	\draw[dotted,gray] (7,-1) -- (7,2);
	\draw[dotted,gray] (8,-1) -- (8,2);
	\draw[scale=2.0, domain=0:4, smooth, variable=\x, blue] plot ({\x}, {0.1+0.3*sin(180*\x)+0.05*\x*\x - 0.2*cos(360*(\x+1))});
	\filldraw [black] (0,-1) circle (1.3pt);
	\draw (0,-1) -- (1,0);
	\filldraw [black] (1,1) circle (1.3pt);
	\draw (1,1) -- (2,2);
	\filldraw [black] (2,-1) circle (1.3pt);
	\draw (2,-1) -- (3,0);
	\filldraw [black] (3,0) circle (1.3pt);
	\draw (3,0) -- (4,1);
	\filldraw [black] (4,0) circle (1.3pt);
	\draw (4,0) -- (5,1);
	\filldraw [black] (5,1) circle (1.3pt);
	\draw (5,1) -- (6,2);
	\filldraw [black] (6,0) circle (1.3pt);
	\draw (6,0) -- (7,1);
	\filldraw [black] (7,1) circle (1.3pt);
	\draw (7,1) -- (8,2);
	\draw[decorate, decoration={brace,raise=3pt}] (3,2) -- (4,2)
		node [midway, above=4pt] {$\tau$};
\end{tikzpicture}
\end{center}
\caption[Construction of discretized time delays]{\label{fig:discretizaiton-example} An illustration of how to uniformly approximate uniformly continuous time delays $h_i(t)$ (blue) by a function $\hat h_i(t)$ in $\mathrm{LI}$ (black), by setting $\hat h_i(n\tau)=\tau \floor{h_i(n\tau)/\tau}$.
The gridlines indicate multiples of $\tau$ in either direciton.
As $\tau\to 0$, the approximation $\hat h_i$ converges in the uniform norm to $h_i$.}
\end{figure}

To prove stability of the global linearization \eqref{eqn:dde-linearized-v2} for a {given} choice of time delays $h$, we introduce {a family} of operators $\Sigma(t,h)$ defined by
\[
\Sigma(t,h)=\set{R^{\eta}_t:\,\text{$\eta:[0,t]\to\R^d$ where $\eta(s)=h(s+nt)$ for some $n\in \N$}}.
\]
Observe that for any $n$, $R^h_{nt}$ may be written as the $n$-fold composition $R^{h_n}_t\circ\ldots\circ R^{h_1}_t$, where $h_i:[0,t]\to \R^d$ is given by $h_i(s)=h(s+(n-1)t)$; in particular, each $R^{h_i}_t\in \Sigma(t,h)$. Thus, {rather than} directly studying $R^h_{nt}$ for large $n$, we may {instead analyze} arbitrary products of elements in $\Sigma(t,h)$, which forms a \emph{discrete-time switched linear system}. The asymptotic behavior of such systems is {governed} by the \emph{joint spectral radius} (JSR), defined as follows. 

If $\mathscr M\subseteq L(X)$ is a collection of bounded operators on the Banach space $X$, let 
\[
\mathscr M^n=\set{A_1\cdots A_n:\,A_i\in \mathscr M}
\]
be the collection of all $n$-fold products of elements of $\mathscr M$. The \emph{joint spectral radius} $\hat\rho(\mathscr M)$, first introduced by Rota and Strang in \cite{rotastrang1960jsr}, is the value
\[
\hat\rho(\mathscr M)=\limsup_{n\to\infty}\sup_{A\in \mathscr M^n} \norm{A}^{1/n}.
\]
This characterizes the maximal exponential growth rate of products of the operators in $\mathscr M$. 
In particular, to show stability of Equation \eqref{eqn:dde-linearized-v2} for a particular choice of $h$, it is sufficient to show that $\hat\rho(\Sigma(t,h))<1$ for any $t>0$ (see \Cref{thm:almost-mr} for full details).

A closely related quantity that is useful to understand $\hat\rho(\mathscr M)$ is the \emph{generalized spectral radius} (GSR) $\rho_*(\mathscr M)$, defined by
\[
\rho_*(\mathscr M)
=
\sup_{n\geq 1} \sup_{A\in \mathscr M^n} \rho(A)^{1/n}.
\]
The JSR and GSR are related by the (generalized) Berger-Wang equation:

\begin{theorem*}[Generalized Berger-Wang {\cite[Theorem 4.5]{shulman2008arxiv}, \cite{shulman2012book}}]
Suppose that either
\begin{enumerate}[label=(\roman*)]
\item $\mathscr M\subseteq M_{n\times n}(\F)$ is a bounded set of $n\times n$ matrices, or
\item $\mathscr M\subseteq L(X)$ is a precompact set of linear operators on a Banach space $X$, and for some $n\geq 1$, every $A\in \mathscr M^n$ is a compact operator,
\end{enumerate}
then $\hat\rho(\mathscr M)=\rho_*(\mathscr M)$.

\end{theorem*}

If (i) holds, the identity $\hat\rho(\mathscr M)=\rho_*(\mathscr M)$ is usually referred to as the \textit{Berger–Wang formula}. While the joint spectral radius $\hat\rho(\mathscr M)$ is useful for describing the dynamical behavior of the switched system generated by $\mathscr M$, the generalized spectral radius $\rho_*(\mathscr M)$ is easier to compute in our setting. Moreover, since its definition involves only suprema, $\rho_*(\mathscr M)$ is better behaved under limits; for example, see the proof of \Cref{thm:jsr-h-bound}.

We now begin our analysis of $\hat\rho(\Sigma(t,h))$.
We will first consider the case where $h\in \mathrm{LI}$.
For each $t>0$, define the set
\[
\Sigma_\tau(t)=\set{R^h_{[0,t]}:\,h\in \linfuns_\tau([0,t])},
\]
to be the set of flow map operators for the linear DDE given in Equation \eqref{eqn:dde-linearized-v2} for each choice of time delay belonging to $\mathrm{LI}_\tau([0,t])$.
Unless otherwise stated, we will treat each $R^h_t$ as acting on the set of Lipschitz functions $C^{0,1}([0,t])$.
The structure of the functions in $\mathrm{LI}_\tau$ gives rise to the following nice relation betweeen the sets $\Sigma_\tau(t)$ for certain values of $t$:
\begin{proposition}[Composition and Decomposition]\label{thm:operator-decomposition}
    If $R^{h_1}\in \Sigma_\tau(n\tau)$ for $n\in \N$ and $R^{h_2}\in \Sigma_\tau(t)$ for any $t>0$, then the composition $R^{h_2}\circ R^{h_1}\in \Sigma_\tau(t+n\tau)$.

    Moreover, for $t>0$, if $n=\floor{t/\tau}$ then any $R^h\in \Sigma_\tau(t)$ may be decomposed (uniquely) into
    \[
    R^h=R^{h'}\circ R^{h_n}\circ\ldots\circ R^{h_1},
    \]
    where $R^{h_i}\in \Sigma_\tau(\tau)$ for each $1\leq i \leq n$, and $R^{h'}\in \Sigma_\tau(t-n\tau)$.
\end{proposition}
\begin{proof}
    If $R^{h_1}\in \Sigma_\tau(n\tau)$ and $R^{h_2}\in \Sigma_\tau(t)$, then note that the concatenation $h_1h_2\in \mathrm{LI}([0,t+n\tau])$, as the points $\tau \N\cap [0,t]$ for $h_2$ get translated to the points $(\tau \N\cap [0,t])+n\tau=\tau \N\cap [n\tau,t+n\tau]$.

    Now, if $R^h\in \Sigma_\tau(t)$ and $n=\floor{t/\tau}$, we may view the function $h\in \mathrm{LI}([0,t])$ as the concatenation of $n$ elements $h_1,\ldots, h_n$ of $\mathrm{LI}([0,\tau])$ followed by an element $h'$ of $\mathrm{LI}([0,t-n\tau])$.
    Splitting $h$ in this way gives the claimed decomposition.
\end{proof}

The primary significance of the decomposition in \Cref{thm:operator-decomposition} is that $\Sigma_\tau(t)$ can be understood purely in terms of products of the set $\Sigma_\tau(\tau)$.
Hence, our present objective is now to bound the JSR $\hat\rho(\Sigma_\tau(\tau))$, which in order to do we will be begin by studying the spectral properties of the elements of $\Sigma_\tau(\tau)$.

Fix some $R^h_\tau\in \Sigma_\tau(\tau)$, with time delays $h(t)=(h_1(t),\ldots,h_r(t))$ given by
\[
h_i(t)=n_i\tau +t,\quad\quad t\in[0,\tau).
\]
Let $x(t)$ be a solution to the corresponding DDE \eqref{eqn:dde-linearized-v2} with initial condition $\phi(t)$ defined on $[-T,0]$.
Then on the interval $[0,\tau]$, we have that $x$ satisfies
\begin{align*}
x'(t)&=\mathcal M_0 x(t)+\sum_{i=1}^r \mathcal M_i x(t-h_i(t))
\\&=\mathcal M_0 x(t)+\sum_{i=1}^r \mathcal M_i x(-n_i\tau)
\\&=\mathcal M_0 x(t)+\sum_{i=1}^r \mathcal M_i \phi(-n_i\tau).\eqlabel{eqn:linearized-interval}
\end{align*}
Hence $x(t)$ satisfies this inhomogeneous linear ODE \eqref{eqn:linearized-interval} on the interval, which can be solved using variation of parameters to give the solution
\begin{align*}
x(t)
&=
e^{\mathcal M_0 t} \phi(0)+e^{\mathcal M_0 t}\int_0^t e^{-\mathcal M_0 s}\sum_{i=1}^r \mathcal M_i \phi(-n_i\tau)\,ds
\\&=
e^{\mathcal M_0 t} \phi(0)+\mathcal M_0^{-1}(e^{\mathcal M_0 t}-I)\sum_{i=1}^r \mathcal M_i \phi(-n_i\tau). \eqlabel{eqn:discretization-formula}
\end{align*}

{In particular,} the solution on this interval is completely determined by the finite {collection} of points $\{\phi(-n_i\tau)\}$. From Equation \eqref{eqn:discretization-formula} we can derive the formula
\[
R^h_\tau[\phi](0)=x(\tau)=e^{\mathcal M_0 \tau} \phi(0)+\mathcal M_0^{-1}(e^{\mathcal M_0 \tau}-I)\sum_{i=1}^r \mathcal M_i \phi(-\tau n_i).
\]
Additionally, {it follows for any $m>0$ that} $R^h_\tau[\phi](-m\tau)=\phi(-(m-1)\tau)$.
We summarize these relations as the following linear system:
\begin{equation}
\pmat{
R^h_\tau[\phi](0) \\
R^h_\tau[\phi](-\tau) \\
R^h_\tau[\phi](-2\tau) \\
\vdots \\
R^h_\tau[\phi](-n_\tau\tau)
}
=
\pmat{
e^{ M_0 \tau} + N_0 & N_1 &  \cdots & N_{n_\tau-1} & N_{n_\tau} \\
I & 0 & \cdots  & 0& 0 \\
0 & I & \cdots  & 0& 0 \\
\vdots& \vdots & \ddots & \vdots & \vdots \\
0 & 0 & \cdots & I & 0
}
\pmat{
\phi(0) \\
\phi(-\tau) \\
\phi(-2\tau) \\
\vdots \\
\phi(-n_\tau\tau)
}
\label{eqn:discretized-matrices}
\end{equation}
where
\begin{equation}\label{eqn:discretized-matrices-index}
N_m=\mathcal M_0^{-1}(e^{\mathcal M_0 \tau}-I)\sum_{\substack{1\leq i \leq r\\n_i=m}} \mathcal M_i.
\end{equation}
Let $\mathscr M_\tau$ be the collection of all matrices of the form in \eqref{eqn:discretized-matrices}.
These matrices have a one-to-one correspondence with $\Sigma_\tau(\tau)$ (or equivalently $\mathrm{LI}_\tau([0,\tau])$):
a matrix $A\in \mathscr M_{\tau}$ corresponds to an operator $R^{h}\in \Sigma_\tau(\tau)$
via letting $n_i$ in \eqref{eqn:discretized-matrices-index} be the integer such that $h_i(0)=n_i\tau$.

We have the following theorem concerning the collections of matrices $\mathscr M_\tau$:
\begin{theorem}[Finite-dimensional JSR bound]\label{thm:main-jsr-bound}
    Fix any $t>0$.
    For any $\epsilon>0$, the collections of matrices $\mathscr M_{t/n}$ satisfy
    \[
    \hat\rho(\mathscr M_{t/n})\leq 1-\frac tn(\gamma_0-\epsilon)
    \]
    for all sufficiently large $n$,
    where $\gamma_0>0$ is given by
    \begin{equation}\tag{\ref{eqn:growth-rate2}}
    \gamma_0=\inf\set{\gamma\in \R:\,\det\paren{\gamma I+M_0+e^{\gamma T}\sum_{i=1}^rM_i
    }=0}.
    \end{equation}
\end{theorem}
The proof of \Cref{thm:main-jsr-bound} is graph-theoretic in nature and has important connections to the theory of delay-independent stability for discrete-time systems.
However, due to its length, we postpone the proof to the following Section 5.

We now seek to connect the collection of matrices $\mathscr M_\tau$ to the collection of operators $\Sigma_\tau(\tau)$.
Let $\pi_\tau:C([-T,0];\R^d)\to \R^{(n_\tau+1)d}$ be the evaluation map 
\[
\pi_\tau(\phi)=(\phi(0),\phi(-\tau),\ldots,\phi(-n_\tau\tau)).
\]
Equation \eqref{eqn:discretized-matrices} shows that $\pi_\tau$ is a semiconjugacy between $R_\tau^h$ and its corresponding matrix $A\in\mathscr M_\tau$.
The semiconjugacy $\pi_\tau$ has particularly nice properties involving the kernels and spectra of the elements of $\Sigma_\tau(\tau)$, which we record in the following two propositions.

\begin{proposition}\label{thm:discretization-kernel}
For any $t\geq T$ and $R\in \Sigma_\tau(t)$,
$
\ker(\pi_\tau)\subseteq \ker(R).
$
\end{proposition}
\begin{proof}
Let $R\in \Sigma_\tau(t)$ for $t\geq T=n_\tau\tau$, and suppose $\phi\in \ker(\pi_\tau)$.
Using \Cref{thm:operator-decomposition}, we may write $R= R'\circ R_1\circ \cdots \circ R_{n_\tau}$ for $R_i\in \Sigma_\tau(\tau)$ and $R'\in \Sigma_\tau(t-n_\tau\tau)$.
The computation \eqref{eqn:discretization-formula} shows that $R_{n_\tau}[\phi]$ is identically zero on the interval $[-\tau,0]$.
As it is still in the kernel of $\pi_\tau$ as well, repeating this recursively implies that $R_1\circ \cdots \circ R_{n_\tau}[\phi]=0$, so $\phi\in \ker(R)$.
\end{proof}

\begin{lemma}\label{thm:discretization-spectrum}
Let $R_1,\ldots,R_m\in \Sigma_\tau(\tau)$ with corresponding matrices $A_1,\ldots, A_m\in \mathscr M_\tau$.
Then
\[
\sigma(R_1\cdots R_m)
=
\sigma(A_1\cdots A_m)\cup\set{0}.
\]
In particular, $\rho(R_1\cdots R_m)=\rho(A_1\cdots A_m)$.
\end{lemma}
The proof of \Cref{thm:discretization-spectrum} is given in Appendix A.
\Cref{thm:discretization-spectrum} has the following consequence:
\begin{corollary}
    For any $\tau>0$, $\rho_*(\mathscr M_\tau)=\rho_*(\Sigma_\tau(\tau))$.
\end{corollary}
\begin{proof}
    Immediate from the definitions of $\rho_*(\mathscr M_\tau)$ and $\rho_*(\Sigma_\tau(\tau))$ and applying \Cref{thm:discretization-spectrum}.
\end{proof}

We are now ready to give an upper bound on $\hat\rho(\Sigma(t,h))$.
The key idea is that the elements of $\Sigma(t,h)$ may be approximated by elements of $\Sigma_\tau(t)$, by approximating the time delays $h$ (see \Cref{thm:continuity}).
This computation requires the following technical lemma in order to satisfy the precompactness requirement of the generalized Berger-Wang theorem:
\begin{lemma}\label{thm:splitting-value}
    If $h\in\overline{\mathrm{LI}}$, then there is $t>0$ such that the following holds:
    \begin{enumerate}[label=(\roman*)]
        \item For some integers $m_n\to\infty$, there are time delays $\hat h^n\in \mathrm{LI}_{t/m_n}$ with $\hat h^n\to h$ in the $L^\infty$-norm as $n\to\infty$.
        \item Labeling $h^{k;t}:[0,t]\to \R^d$ as the funtion given by $h^{k;t}(s)=h(s+kt)$, the set
        \[
        \set{h^{k;t}:\,k\in \N}
        \]
        is precompact in $L^\infty([0,t];\R^d)$.
    \end{enumerate}
\end{lemma}
The proof of \Cref{thm:splitting-value} is given in the Appendix.
We now calculate a bound on $\hat\rho(\Sigma(t,h))$.

\begin{lemma}\label{thm:jsr-h-bound}
    If $h\in \overline{\mathrm{LI}}$, then there is $t>0$ such that
\[
\hat\rho(\Sigma(t,h))\leq e^{-\gamma_0 t}
\]
    where $\gamma_0>0$ is as in \Cref{thm:main-jsr-bound}.
\end{lemma}

\begin{proof}
    Let $t>0$ satisfy the conclusion of \Cref{thm:splitting-value}.
    For every $\epsilon>0$, \Cref{thm:main-jsr-bound} implies for all sufficiently large $n$ that
    \[
    \hat\rho(\mathscr M_{t/n})\leq 1-\frac{t}{n}(\gamma_0-\epsilon)
    \]
    Since each of the collections $\mathscr M_{t/n}$ are clearly bounded (c.f. Equation \eqref{eqn:discretized-matrices}, the Berger-Wang equation \cite{berger1992} implies the same bounds holds for $\rho_*(\mathscr M_{t/n})$.
    Next, since by the decomposition in \Cref{thm:operator-decomposition} we have $\Sigma_{t/n}(t)=\Sigma_{t/n}(t/n)^n$.
    Thus we compute
    \begin{align*}
    \rho_*(\Sigma_{t/n}(t))
    &\leq \rho_*(\Sigma_{t/n}(t/n))^n
    \\&=
    \rho_*(\mathscr M_{t/n})^n
    \\&\leq\paren{1-\frac{t}{n}(\gamma_0-\epsilon)}^n
    \\&=\exp(-(\gamma_0-\epsilon)t+O(t^2n^{-1}))
    \\&\leq\exp(-(\gamma_0-2\epsilon)t)
    \end{align*}
    for sufficiently large $n$ depending on $\epsilon$.

    Now, let $R^{h_1},\ldots, R_{h_n}\in \Sigma(t,h)$.
    Since $t$ satisfies the conclusion of \Cref{thm:splitting-value}, there are $\hat h_{1,k},\ldots,\hat h_{m,k}\in \mathrm{LI}_{t/n_k}([0,t])$ where $n_k\to\infty$ and $\hat h_{i,k}\to h_i$ as $k\to\infty$.
    By continuity of the spectral radius, as $k\to\infty$,
    \[
    \rho(R^{h_{m}}\cdots R^{h_{1}})
    =\lim_{k\to\infty} \rho(R^{\hat h_{m,k}}\cdots R^{\hat h_{1,k}})
    \leq \lim_{k\to\infty}\rho_*(\mathscr M_{t/n_k})^m
    \leq \exp(-(\gamma_0-2\epsilon)mt),
    \]
    so
    \[
    \rho_*(\Sigma(t,h))
    =
    \sup_{m\geq 1}\sup_{R^{h_i}\in \Sigma(t,h)}\rho(R^{h_{m}}\cdots R^{h_{1}})^{1/m}
    \leq \exp(-(\gamma_0-2\epsilon)t).
    \]
    As this holds for every $\epsilon>0$, we thus have $\rho_*(\Sigma(t,h))\leq e^{-\gamma_0 t}$.
    
    Next, as $\set{h^{n;t}}$ is precompact due to \Cref{thm:splitting-value}, the set $\Sigma(t,h)$ is the image under the continuous map $h\mapsto R^h_t$ of the precompact set $\set{h^{n;t}}$, and is itself precompact.
    Finally, note that every operator $R\in\Sigma(t,h)^m$ is of the form $R=R^{\tilde h}_{tm}$ for some time delay $\tilde h$; hence if $tm\geq T$, each such operator is compact.
    Thus the generalized Berger-Wang equation applies, which implies
    \[
    \hat\rho(\Sigma(h,t))=\rho_*(\Sigma(t,h))\leq e^{-\gamma_0 t}
    \]
    as desired.
\end{proof}

We now are ready to show the following special case of Main Results 1 and 2, from which all three main results follow.
\begin{theorem}[Stability of the Global Linearization]\label{thm:almost-mr}
    Main Results 1 and 2 hold for the global linearization
    \[
    x'(t)=M_0x(t)+\sum_{i=1}^r M_i x(t-h_i(t))
    \]
    for any time delays $h(t)$ satisfying Assumption 1.
    In particular, the solution $x = 0$ is globally exponentially stable. Moreover, every solution $x(t)$ satisfies
\begin{equation}\label{eqn:almostmr-statement}
\norm{x(t)} \le C e^{-(\gamma_0-\varepsilon)t}
\sup_{s\le 0}\norm{x(s)}
\end{equation}
for any $\epsilon>0$ and some constant $C>0$, where $\gamma_0>0$ is defined in Equation~\eqref{eqn:growth-rate2}.
\end{theorem}

\begin{proof}
    It suffices to show that Equation \eqref{eqn:almostmr-statement} holds.
    Let $t_0$ be a splitting value for $h$, whose existence is guaranteed by \Cref{thm:splitting-value}; then \Cref{thm:jsr-h-bound} implies $\hat\rho(\Sigma(h,t))\leq e^{-\gamma_0 t}$,
    so for any $\epsilon>0$, there is $C>0$ such that for sufficiently large $m$ and all $R^{h_1},\ldots, R^{h_m}\in \Sigma(h,t_0)$,
    \[
    \norm{R^{h_m}\cdots R^{h_1}}^{1/m}\leq \exp((-\gamma_0+\epsilon)t_0).
    \]
    Now, for any $t\geq 0$, write $t=mt_0+t'$ for $0\leq t'<t_0$.
    We may decompose $h|_{[0,t]}$ as a concatenation of functions $h^1,\ldots, h^m,h'$
    where $h^i:[0,t_0]\to \R^d$, $h':[0,t']\to\R^d$.
    This gives a decomposition $R^h_t=R^{h'}_{t'}\circ R^{h_m}_{t_0}\circ\ldots\circ R^{h_1}_{t_0}$.
    Using continuity of solutions and the joint spectral radius bound, we thus have
    \[
    \norm{R^h_t}_{C^{0,1}}
    \leq
    \norm{R^{h'}_{t'}}\norm{R^{h_m}_{t_0}\ldots R^{h_1}_{t_0}}
    \leq e^{Lt'}e^{(-\gamma_0+\epsilon)t_0m}
    \leq
    Ce^{(-\gamma_0+\epsilon)t}
    \]
    for $C=\exp\paren{(L+\gamma_0-\epsilon)T}$.
    Let $x(t)$ be the solution to Equation \eqref{eqn:dde-linearized-v2} with initial condition $x(t)=\phi(t)$, $t\leq 0$, where $\phi\in C([-T,0])$.
    Observe that for $t\geq T$ we may write
    \[
    x(t)=R^h_t[\phi](0)=R^h_{[T,t]}[R^h_T[\phi]](0);
    \]
    note furthermore that $R^h_T[\phi]$ is Lipschitz continuous and $\norm{R^h_T[\phi]}_{C^{0,1}}\leq C'\norm{\phi}_{C^0}$ by \Cref{thm:continuity}.
    Hence
    \[
    \norm{x(t)}
    =
    \norm{R^h_t[\phi](0)}
    \leq
    \norm{R^h[\phi]}_{C^{0,1}}
    \leq
    \norm{R^h_{[T,t]}}_{C^{0,1}}\norm{R^h_T[\phi]}_{C^{0,1}}
    \leq CC' e^{(-\gamma_0+\epsilon)t}\norm{\phi}_{C^0}
    \]
    which shows the theorem.
\end{proof}

Combining this result with \Cref{thm:bound-by-linear}, we are now able to prove Main Results \ref{mr1}--\ref{mr3}.
We begin with Main Result 3:
\begin{proof}[Proof of Main Result 3]
    Suppose that the condition \eqref{eqn:stabcond} holds, and let $x_1(t),x_2(t)$ be any solutions to the system \eqref{eqn:dde}.
    Let $y(t)$ be the solution to the global linearization \eqref{eqn:dde-linearized} with initial condition $y(t)=\abs{x_1(t)-x_2(t)}$ for $t\leq 0$.
    Then, \Cref{thm:bound-by-linear} implies 
    \begin{equation}\label{eqn:mr-proof-step1}
        \abs{x_1(t)-x_2(t)}\leq y(t)
    \end{equation} for all $t\geq 0$.
    As uniformly continuous time delays satisfy Assumption 1, \Cref{thm:almost-mr} implies that \begin{equation}\label{eqn:mr-proof-step2}
        \norm{y(t)}\leq Ce^{-(\gamma_0-\epsilon)t}\sup_{s\leq 0}\norm{y(s)}.
    \end{equation}
    Choosing any of the standard $p$-norms for the norm on $\R^d$, it is the case that $\norm{\abs{v}}=\norm{v}$ and $0\leq u \leq v$ implies $\norm{u}\leq \norm{v}$
    Hence, combining \eqref{eqn:mr-proof-step1} and \eqref{eqn:mr-proof-step2}
    implies that
    \[
    \norm{x_1(t)-x_2(t)}\leq Ce^{-(\gamma_0-\epsilon)t}\sup_{s\leq 0}\norm{x_1(s)-x_2(s)}
    \]
    showing Main Result 3.
\end{proof}
Main Results 1 and 2 follow as a corollary of Main Result \ref{mr3}.
\begin{proof}[Proof of Main Results 1 and 2]
    Let $x_0\in \R^d$ be the equilibrium point of the system \eqref{eqn:dde}.
    Applying Main Result 3 with $x_1(t)=x(t)$ any solution and $x_2(t)=x_0$ the equilibrium solution, we have
\[
\norm{x(t)-x_0}\leq Ce^{-(\gamma_0-\epsilon)t}\sup_{s\leq 0}\norm{x(s)-x_0}
\]
for any $\epsilon>0$ and some $C>0$,
which shows Main Result 2.
    This moreover shows that $x_0$ is globally exponentially stable, showing Main Result 1.
\end{proof}
Note that the above proofs do not depend in any essential way on the time delays $h$ being uniformly continuous, other than requiring $h\in\overline{\mathrm{LI}}$; hence additionally Main Results \ref{mr1}--\ref{mr3} hold for any $h\in \overline{\mathrm{LI}}$.
For example, this more general framework includes the discontinuous time delays $h(t)=t\mod 2$ in Example 2.1(d).

\section{The Joint Spectral Radius of Discretized Systems}\label{sec:gsr-bound}
With the proof of Main Results~\ref{mr1}--\ref{mr3} reduced to proving \Cref{thm:main-jsr-bound}, it remains to bound the joint spectral radius of the matrix family $\mathscr M_\tau$ arising from the discretization of the global linearization. These matrices encode the delayed interactions in the original system, but their structure also makes direct spectral analysis a challenge.

To address this, we develop a novel approach combining two complementary tools: a characterization of the joint spectral radius via row-independent closures, and the graph-theoretic technique of isoradial reduction. This framework reduces the problem to the analysis of simpler matrix families, ultimately yielding the desired bound.

\Cref{thm:main-jsr-bound} considers the joint spectral radius for the collection of all matrices of the form
\begin{equation}
\label{eqn:discretized-matrices-2}
A=
\pmat{
e^{ M_0 \tau} + N_0 & N_1 &  \cdots & N_{n_\tau-1} & N_{n_\tau} \\
I & 0 & \cdots  & 0& 0 \\
0 & I & \cdots  & 0& 0 \\
\vdots& \vdots & \ddots & \vdots & \vdots \\
0 & 0 & \cdots & I & 0
},\quad\quad
N_m=\mathcal M_0^{-1}(e^{\mathcal M_0 \tau}-I)\sum_{\substack{1\leq i \leq r\\n_i=m}} \mathcal M_i
\end{equation}
where $M_0,M_1,\ldots, M_r$ are the stability matrices of Equation \eqref{eqn:dde}, $n_\tau = \floor{T/\tau}$, and $n_1,\ldots, n_r$ are integers in $\set{0,1,\ldots,n_\tau}$ determined by the time delays of the DDE \eqref{eqn:dde-linearized-v2}.
The proof of this theorem has interesting connections with the problem of delay-independent stability. 

To carry out the proof, we first introduce two tools. The first is a result of Blondel and Nesterov \cite{blondel2010} characterizing the joint spectral radius for certain classes of matrices. This requires the notion of a \emph{row-independent closure}, defined for $\mathscr{M}\subseteq M_{n\times n}(\F)$ by
\[
\mathrm{RIC}(\mathscr M)
=
\Bigl\{
A\in M_{n\times n}(\F):
\text{for each } i,\ \text{row } i \text{ of } A \text{ equals row } i \text{ of some } A_i\in\mathscr M
\Bigr\}.
\]
The result is as follows.

\begin{theorem*}[{\cite[Theorem 2]{blondel2010}}]
    If $\mathscr M\subseteq M_n(\R)$ is a compact set of non-negative matrices that satisfies $\mathrm{RIC}(\mathscr M)=\mathscr M$, then
    \[
    \hat\rho(\mathscr M)=\sup_{A\in \mathscr M}\rho(A).
    \]
\end{theorem*}
We will use this result in the following form.
\begin{corollary}\label{thm:ric-bound}
    For a bounded set $\mathscr M\subseteq M_n(\R)$ of nonnegative matrices,
    \[
    \hat\rho(\mathscr M)\leq \sup_{A\in \mathrm{RIC}(\overline{\mathscr M})}\rho(A).
    \]
\end{corollary}
\begin{proof}
    From the definition of $\hat\rho$ it is clear that it is nondecreasing under inclusion; hence $\hat\rho(\mathscr M)\leq \hat\rho(\mathrm{RIC}(\overline{\mathscr M}))$. The result then follows by the above-cited theorem.
\end{proof}

The second tool we use is the graph-theoretic technique of \emph{isoradial reduction}, a special case of \emph{isospectral reduction} that produces a lower-dimensional matrix while preserving spectral properties. To define it, let $B\in M_{n\times n}(\F)$ and set $N=\{1,\ldots,n\}$. For subsets $R,C\subseteq N$, denote by $B_{RC}$ the $\lvert R\rvert\times\lvert C\rvert$ submatrix of $B$ formed by selecting rows in $R$ and columns in $C$. For $S\subseteq N$, write $\bar S = N\setminus S$.

The isoradial reduction of $B$ over a nonempty subset $S\subseteq N$ is defined by
\begin{equation}\label{eqn:isoradial-reduction}
\mathcal I_S(B)
=
B_{SS}-B_{S\bar S}(B_{\bar S\bar S}-\rho(B) I)^{-1}B_{\bar S S}.
\end{equation}

The isoradial reduction $\mathcal I_S(B)$ is not defined for every nonnegative matrix $B\in M_{n\times n}(\R)$ and subset $S\subseteq N$, since the inverse in \eqref{eqn:isoradial-reduction} may fail to exist. When it is defined, the following result shows that it preserves the spectral radius.
\begin{theorem*}[Isoradial reduction {\cite[Theorem 4.4]{smith2019}}]
If $\mathcal I_S(B)$ exists, then
\begin{equation}\label{eqn:isoradial}
\rho(\mathcal I_S(B)) = \rho(B).
\end{equation}
\end{theorem*}

A useful consequence of Equation \eqref{eqn:isoradial} is the following result, originally proved as Lemma 3 in \cite{Carter_2022}.

\begin{lemma}\label{thm:delay-matrix-isoradial}
    Let $A\in M_{mn\times mn}(\R)$ be a nonnegative block matrix with $\rho(A)>0$ of the form
\[
A
=
\pmat{
A_0 & A_1 & A_2 & \cdots & A_{n-1} \\
I & 0 & 0 & \cdots & 0 \\
0 & I & 0 & \cdots & 0 \\
\vdots & \vdots & \ddots &  & \vdots \\
0 & 0 & \cdots & I & 0
}
\]
where each block $A_i\in M_{m\times m}(\R)$ is square.
Then 
\begin{equation}\label{eqn:delay-reduction}
\rho(A)
=
\rho\paren{
\sum_{i=0}^{n-1} \rho(A)^{-i} A_i
}
\end{equation}
Moreover, $\rho(A)<1$ if and only if $\rho(\sum_{i=0}^{n-1} A_i)<1$, in which case $\rho(\sum_{i=0}^{n-1} A_i)\leq \rho(A)<1$.
\end{lemma}
\begin{proof}
Equation \eqref{eqn:delay-reduction} follows immediately from Equation \eqref{eqn:isoradial} by computing the isoradial reduction of $A$ with respect to the set $S=\set{1,\ldots,m}$; see \cite{Carter_2022} for details.
The final claim follows from monotonicity of the spectral radius in each entry of a nonnegative matrix (see, for instance, \cite[Chapter 8]{HornJohnson2013}). 
\end{proof}

Matrices of this form typically arise from time-delayed discrete-time dynamical systems; see \cite{Carter_2022} for more details. It is interesting to note that the matrices in the collection $\mathscr M_\tau$ are of the same form (see Equation \eqref{eqn:discretized-matrices-2}). Finally, before proving \Cref{thm:main-jsr-bound}, we need the following lemma. Recall that for a square matrix $A$, the spectral abscissa $\alpha(A)=\sup\Re(\sigma(A))$ is the largest real part of any of its eigenvalues.

Matrices of this form typically arise in time-delayed discrete-time dynamical systems; see \cite{Carter_2022} for further details. Notably, the matrices in the collection $\mathscr M_\tau$ have the same structure (cf. \eqref{eqn:discretized-matrices-2}).

To proceed with the proof of \Cref{thm:main-jsr-bound}, we require the following lemma. We first recall that for a square matrix $A$, the spectral abscissa $\alpha(A)=\sup\Re(\sigma(A))$ is the largest real part of its eigenvalues.

\begin{lemma}[label=thm:spec-rad-asymptotic]
Let {$(A_n)_{n\ge 1}$} be a sequence of {square} matrices converging to $A$.
Then, as $n\to\infty$,
\[
\rho\paren{I+\frac1n A_n}=1+\frac1n\alpha(A)+o\paren{\frac1n}.
\]
\end{lemma}

\begin{proof}
We may write $A_n\to A$ as $A_n=A+o(1)$.
Recall that for any complex number $z$, we have
$
\abs{1+z}
=
1+\Re(z)+O(\abs{z}^2)$
as $\abs{z}\to 0$.
We thus have
\begin{align*}
\rho\paren{I+\frac1nA_n}
&=
\sup_{\lambda\in \sigma(A_n)}\abs{1+\frac1n\lambda}
\\&=
\sup_{\lambda\in \sigma(A_n)}\paren{1+\frac1n\Re(\lambda)+O(n^{-2}\abs{\lambda}^2)}
\\&=
1+\frac1n\alpha(A_n)
+O(n^{-2} \rho(A_n)^2).
\end{align*}
As the sequence $\rho(A_n)$ must be bounded, the big-O term is in fact $O(n^{-2})$.
By continuity of the spectral abscissa, we additionally have $\alpha(A_n)=\alpha(A)+o(1)$.
Combining these asymptotics gives
\[
\rho\paren{I+\frac1nA_n}
=
1+\frac1n\alpha(A_n)
+O(n^{-2})
=
1+\frac1n \alpha(A)+o(n^{-1})
\]
as desired.
\end{proof}

We now prove \Cref{thm:main-jsr-bound}.
\begin{proof}[Proof of \Cref{thm:main-jsr-bound}]
Using \Cref{thm:ric-bound}, we have that
\begin{equation}\label{eqn:jsr-bound-step1}
    \rho(\mathscr M_{t/n})\leq \sup_{A\in \mathrm{RIC}(\mathscr{M}_{t/n})}\rho(A).
\end{equation}
We begin by characterizing $\mathrm{RIC}(\mathscr M_{t/n})$.
Recalling Equation \eqref{eqn:discretized-matrices}, each element $A\in \mathscr M_{t/n}$ is of the form
\[
\pmat{
e^{ M_0 t/n} + N_0 & N_1 &  \cdots & N_{m-1} & N_{m} \\
I & 0 & \cdots  & 0& 0 \\
0 & I & \cdots  & 0& 0 \\
\vdots& \vdots & \ddots & \vdots & \vdots \\
0 & 0 & \cdots & I & 0
}
\]
where $m=\ceil{nT/t}$, and the matrices $N_j$, $0\leq j\leq m$, may be written as
\[
N_j= M_0^{-1}(e^{ M_0 t/n}-I)\sum_{\substack{1\leq i \leq r\\n_i=j}}  M_i
\]
where the integers $n_1,\ldots,n_r\in \set{0,1,2,\ldots,m}$ are determined by the time delays $h$ on $[0,n/t]$, and where $M_0,\ldots,M_r$ are the stability matrices of the system \eqref{eqn:dde}.

Write $B_{i,n}=M_0^{-1}(e^{ M_0 t/n}-I) M_i$ for $1\leq i \leq r$.
To describe the row-independent closure of $\mathscr M_{t/n}$, define the matrices $B_{i,j,n}$ for $1\leq j \leq d$ as the $d\times d$ matrix whose $j$-th row is the $j$-th row of $B_{i,n}$, and all other entries are zero.
Note in each element $B\in \mathscr M_n$, each block $B_{i,n}$ for $1\leq i \leq r$ occurs in exactly one term in the top row of blocks in $B$.
Hence, if we consider an arbitrary element $\tilde{B}\in\mathrm{RIC}(\mathscr{M}_{t/n})$, as it is formed out of the rows of elements of $\mathscr M_{t/n}$, we may write
\begin{equation}\label{eqn:ric-matrix-form}
\tilde{B}=\pmat{
e^{\mathcal M_0 t/n} + \tilde{N}_0 & \tilde{N}_1 & \tilde{N}_2 &  \cdots & \tilde{N}_{m} \\
I & 0 & 0 & \cdots  & 0 \\
0 & I & 0 & \cdots  & 0 \\
\vdots& \vdots && \ddots & \vdots \\
0 & 0 & \cdots & I & 0
}
\end{equation}
where
\[
\tilde{N}_k=\sum_{n_{i,j}=k} B_{i,j,n}
\]
for some indices $n_{i,j}\in \set{0,\ldots,m}$, $1\leq i\leq r$, $1\leq j \leq d$.
Note that each row of each $B_{i,n}$ must occur as a term in some top-row block, but different rows from the same $B_{i,n}$ need not occur in the same block.
For notational convenience, write $B_{0,n}=e^{\mathcal M_0 t/n}$.

Observe that as $\mathscr M_{t/n}$ is a finite set, so is $\mathrm{RIC}(\mathscr{M}_{t/n})$; hence the supremum in \eqref{eqn:jsr-bound-step1} is attained for each $n$.
Let $A_n\in\mathrm{RIC}(\mathscr{M}_n)$ be a sequence of matrices that attain the supremum:
\[
\rho(A_n)=\sup_{A\in \mathrm{RIC}(\mathscr{M}_n)}\rho(A)\geq\rho(\mathscr M_n).
\]
As $A_n\in\mathrm{RIC}(\mathscr{M}_{t/n})$, it may be written in the form of Equation \eqref{eqn:ric-matrix-form}; hence using \Cref{thm:delay-matrix-isoradial} to compute its isoradial reduction with respect to the set $S=\set{1,\ldots,d}$, we have 
\[
\mathcal I_S(A_n)=B_{0,n}+\sum_{i=1}^r\sum_{j=1}^d \rho(A_n)^{-n_{i,j,n}} B_{i,j,n}
\]
and $\rho(A_n)
=
\rho\paren{
\mathcal I_S(A_n)
}
$.
Now consider the matrix 
$B_n
=B_{0,n}+\sum_{i=1}^r\sum_{j=1}^d B_{i,j,n}$; note that this corresponds to the sum of top-row blocks in \Cref{thm:delay-matrix-isoradial} for $A_n$.
Note that $B_{0,n}=e^{M_0{t/n}}=I+\frac tn M_0+O(n^{-2})$, and for $1\leq i \leq r$ we have
\[
B_{i,n}=M_0^{-1}(e^{M_0{t/n}}-I)M_i
=
\frac tn M_i+O(n^{-2}).
\]
Hence we have
\[
B_n
=B_{0,n}+\sum_{i=1}^r\sum_{j=1}^d B_{i,j,n}
=\sum_{i=0}^r B_{i,n}
=I+\frac tn \sum_{i=0}^r M_i+O(n^{-2})
=I+\frac{t}{n}\mathcal M+O(n^{-2})
\]
where $\mathcal M=M_0+\ldots+M_r$.
Since $\alpha(t\mathcal M)=t\alpha(\mathcal M)<0$ by Condition \eqref{eqn:stabcond}, it follows from \Cref{thm:spec-rad-asymptotic} that $\rho(B_n)<1$ for all sufficiently large $n$.
\Cref{thm:delay-matrix-isoradial} then gives that $\rho(B_n)\leq \rho(A_n)<1$ for all such $n$.
Moreover, as clearly $\rho(B_n)\to 1$ as $n\to\infty$, we must have $\rho(A_n)\to 1$ also.

For simplicity of notation, label $M_D=\sum_{i=1}^r M_i$, i.e. all of the delayed parts.

Define the sequence $\gamma_n$ by
\[
\gamma_n=\frac nt (1-\rho(A_n)),
\]
so that $
\rho(A_n)=1-\frac{\gamma_n t}{n}$.
If $\gamma_n\to \infty$, it is clear that the claim of the theorem holds for any $\gamma_0>0$; hence it suffices to consider the case where $\gamma_n\not\to\infty$.
Let
\[
\gamma_*=\sup\set{\gamma>0:\,\text{$\rho(A_n)\leq 1-\frac{\gamma t}{n}$ for all sufficiently large $n$}}=\liminf_{n\to\infty} \gamma_n.
\]
Note $\gamma_n>0$ for all $n$; as moreover $\gamma_n\not\to\infty$, $\gamma_*$ is well-defined and nonnegative.

To complete the proof, it suffices to show that $\gamma_*\geq \gamma_0>0$, where we recall $\gamma_0$ is given by
\begin{equation}\label{eqn:growth-rate-ver2}
\gamma_0=\inf\set{\gamma\in\R:\,
    \det\paren{\gamma I+M_0+e^{\gamma T}M_D
    }=0.
    }
    \end{equation}
First we show that $\gamma_0$ is the unique real root of the function
    \[
    g(\lambda)=\lambda+\alpha\paren{M_0+e^{\lambda T}M_D}.
    \]
    Note that $M_0+e^{\lambda T}M_D$ is a Metzler matrix for any $\lambda\in \R$, and hence its spectral abscissa $\alpha(M_0+e^{\lambda T}M_D)$ is monotonic nondecreasing in each entry of the matrix; in particular, the abscissa is nondecreasing in $\lambda$.
    Hence $g(\lambda)$ is a strictly increasing function and has at most one real root.
    Next, monotonicity also implies that $g(\lambda)\geq \lambda+\alpha(M_0)$, so in particular $g(\lambda)\to\infty$ as $\lambda\to\infty$.
    Finally, observe that $g(0)=\alpha(M_0+M_D)<0$ by assumption.
    Hence $g(\lambda)$ has a unique real root $\lambda_0$, and $\lambda_0>0$.

    Now, note that we may also write
\[
    g(\lambda)=\alpha\paren{\lambda I+M_0+e^{\lambda T}M_D};
\]
hence the matrix $\lambda_0I+M_0+e^{\lambda_0T}M_D$ has a zero eigenvalue, so $\det(\lambda_0I+M_0+e^{\lambda_0T}M_D)=0$.
Next, observe that if $\lambda<\lambda_0$, then we have $g(\lambda)<0$, so
\[
\alpha(M_0+e^{\lambda T}M_D)<-\lambda;
\]
in particular, $-\lambda$ cannot be an eigenvalue of $M_0+e^{\lambda T}M_D$, and thus $\det(\lambda I + M_0+e^{\lambda T}M_D)\neq 0$.
Thus $\gamma_0=\lambda_0>0$ is the unique root of $g(\lambda)$.

Now, take a convergent subsequence $\gamma_{n_k}\to \gamma_*$, which we will relabel as $\gamma_k$.
Along this subsequence, we have
\begin{align*}
\gamma_k&=
\frac {n_k}t (1-\rho(A_k))
\\&
=
-\frac {n_k}t \alpha(\mathcal I_S(A_k)-I)
\\&=
-\alpha\paren{
\frac {n_k}t\paren{B_{0,k}-I}
+\sum_{i=1}^r
\sum_{j=1}^d \rho(A_k)^{-n_{i,j,k}} \frac nt B_{i,j,k}
}
\end{align*}
Now, note that $0\leq n_{i,j,k}\leq {n_k T}/{t}$; since $\rho(A_k)<1$, we thus have $\rho(A_k)^{-n_{i,j,k}}\leq \rho(A_k)^{- n_k T/t}$.
Using that $\mathcal I_S(A_k)-I$ is a Metzler matrix and hence $\alpha(\mathcal I_S(A_k)-I)$ is monotonic in the matrix's entries, we thus have
\begin{align*}
\gamma_k
&\geq
-\alpha\paren{
\frac {n_k}t\paren{B_{0,k}-I}
+\rho(A_k)^{-n_kT/t}\sum_{i=1}^r\frac {n_k}t B_{i,k}
}
\\&=-\alpha( M_{0,k}+e^{\gamma_k T}M_{D,k})
\end{align*}
where 
\[
M_{0,k}=
\frac {n_k}t\paren{B_{0,k}-I},
\quad\quad
M_{D,k}=\rho(A_k)^{-n_kT/t} e^{-\gamma_k T}\sum_{i=1}^r \frac {n_k}t B_{i,k}.
\]
Note that from the definition of $B_{0,k}$, it is the case that 
$M_{0,k}=\frac {n_k}t (e^{M_0t/n_k}-I)=M_0+O(n_k^{-1})\to M_0$ as $k\to\infty$.
Since $\rho(A_k)=1-{\gamma_k t}/{n_k}$ we also compute
\[
\rho(A_k)^{- n_k T/t}
=
\exp\paren{
-\frac{n_k T}{t}
\log\paren{
1-\frac{\gamma_k t}{n_k}
}
}
=
\exp\paren{
\gamma_k T
+O(n_k^{-1})
}
=
e^{\gamma_k T}(1+O(n_k^{-1}))
\]
as $\gamma_k$ is bounded.
Hence
\[
M_{D,k}=\rho(A_k)^{-n_kT/t} e^{-\gamma_k T}\frac nt M_0^{-1}(e^{M_0t/n_k}-I)M_D
=
(1+O(n_k^{-1})(M_D+O(n_k^{-1}))
\to M_D
\]
as $k\to\infty$ also.
Hence by continuity, we have
\[
g(\gamma_*)=\gamma_*+\alpha(M_0+e^{\gamma_* T}M_D)
=
\lim_{k\to\infty} \gamma_k+\alpha(M_{0,k}+e^{\gamma_k T}M_{D,k})\geq 0.
\]
The monotonicity of $g$ thus implies that $\gamma_*\geq \gamma_0$ as desired.
\end{proof}

\section{Dynamical Applications and Additional Results}\label{sec:corollaries}

In this section we derive several consequences of our main stability results, clarifying how they specialize to equilibrium, periodic, and local regimes. We first show that autonomous systems satisfying the intrinsic stability condition admit a unique globally exponentially attracting equilibrium that is effectively independent of the system's delay structure (Corollary~\ref{thm:mr-corollary1}). We then extend this to periodic nonautonomous systems, obtaining a unique globally attracting periodic solution with matching period (Corollary~\ref{thm:mr-corollary2}). Finally, we establish a local stability criterion for multistable systems, etc. via linearization at equilibria, yielding delay-independent local exponential stability conditions (Corollary~\ref{thm:mr-corollary3}).

Main Result~\ref{mr1} assumes both intrinsic stability and the existence of an equilibrium. In the autonomous case, however, the existence of an equilibrium follows automatically from the intrinsic stability condition. In particular, the system admits a unique equilibrium, which is globally exponentially stable and independent of the time delays.

\begin{corollary}[Existence and Uniqueness of Equilibria]\label{thm:mr-corollary1}
Suppose that $\alpha(M_0+\cdots+M_r)<0$ holds for the DDE~\eqref{eqn:dde}.
If $f$ is autonomous, then \eqref{eqn:dde} admits a unique equilibrium, which is globally exponentially attracting.
\end{corollary}

\begin{proof}
Observe that it suffices to show the existence of a fixed point for any choice of time delays; hence we may consider constant time delays.
As the time delays are thus also bounded, Main Result 3 implies there are $C>0,\gamma>0$ such that
    \begin{equation}\label{eqn:contraction-mapping}
        \norm{S^h_t[\phi_1]-S^h_t[\phi_2]}_{C^0}\leq Ce^{-\gamma t}\norm{\phi_1-\phi_2}_{C^0}.
    \end{equation}
    If we fix $t$ large enough that $Ce^{-\gamma t}<1$, then Equation \eqref{eqn:contraction-mapping} implies $S^h_t$ is a contraction map.
    Hence the contraction mapping principle implies there is a unique fixed point $\phi_*\in C([-T,0];\R^d)$ of $S^h_t$.
    Note that for any $t'>0$,
    \[
    S^h_t[S^h_{t'}[\phi_*]]=S^h_{t'}[S^h_{t}[\phi_*]]=S^h_{t'}[\phi_*]
    \]
    so $S^h_{t'}[\phi_*]=\phi_*$ by uniqueness of the fixed point.
    However, for $t'<T$ and $s<-t'$, $S^h_{t'}[\phi_*](s)=\phi_*(s+t')$; this implies $\phi_*$ is constant, so $x_*=\phi_*(0)$ is a fixed point of Equation \eqref{eqn:dde}, which is globally exponentially attracting by Main Result 1.
\end{proof}

When both $f$ and the time delays $h(t)$ are periodic with a common period, the system admits a \textit{globally attracting periodic solution with the same period.} Unlike in Corollary~\ref{thm:mr-corollary1}, this periodic solution may depend on the time-delay functions.

\begin{corollary}[Existence of Periodic Solutions]\label{thm:mr-corollary2}
Assume $\alpha(M_0+\cdots+M_r)<0$ and that the time delays $h$ belong to $\overline{\mathrm{LI}}$. If $f$ and $h$ are periodic in $t$ with a common period $P>0$, then \eqref{eqn:dde} admits a unique $P$-periodic solution, which is globally exponentially attracting.
\end{corollary}

\begin{proof}
    As in the proof of \Cref{thm:mr-corollary1} we apply Main Result 3 to obtain the existence of $C>0,\gamma>0$ such that
    \[
    \norm{S^h_t[\phi_1]-S^h_t[\phi_2]}_{C^0}\leq Ce^{-\gamma t}\norm{\phi_1-\phi_2}_{C^0}.
    \]
    Now, choose an integer $k\geq 1$ such that $Ce^{-\gamma kP}<1$;
    then $S^h_{kP}$ is a contraction mapping, and hence has a unique fixed point $\phi_*(t)\in C([-T,0];\R^d)$.
    Since $S^h_{kP}=(S^h_P)^k$ the $k$-fold composition, we have
    \[
    S^h_{kP}[S^h_P[\phi_*]]=S^h_{P}[S^h_{kP}[\phi_*]]=S^h_P[\phi_*],
    \]
    so $S^h_P[\phi_*]=\phi_*$ by uniqueness of the fixed point.
    Letting $x_*(t)$ be the solution with initial condition $\phi_*$, the equality $S^h_P[\phi_*]=\phi_*$ implies that $x_*(t)$ is periodic with period $P$.
    Main Result 3 then implies that this periodic orbit is globally exponentially attracting, which in particular also implies uniqueness.
\end{proof}

Many systems are multistable, having more that one stable fixed point.
In this setting, determining whether stability persists in the presence of time delays requires a local version of the global stability framework developed above.

Let $x_*$ be an equilibrium of \eqref{eqn:dde}. We define the \emph{local stability matrices} $M_0(x_*),\ldots,M_r(x_*)$ by
\[
M_0(x_*) = \sup_{t\geq 0}\mathrm{abs}^*\bigl(D_x f(t,x_*,\ldots,x_*)\bigr),
\quad\quad
M_i(x_*)=\sup_{t\geq 0}\bigl|D_{y_i}f(t,x_*,\ldots,x_*)\bigr|.
\]
These definitions are directly analogous to those of the global stability matrices in Equation~\eqref{eqn:stabmat}. 
As in the global setting (cf. Main Result~\ref{mr1}), these matrices provide a linearized description of the dynamics near $x_*$.

Recall that an equilibrium $x_*$ is \emph{locally stable} if every solution with initial condition in a $C^0$-neighborhood of $x_*$ satisfies $x(t)\to x_*$ as $t\to\infty$. It is \emph{locally exponentially stable} if this convergence occurs at an exponential rate.

\begin{corollary}[Local Intrinsic Stability]\label{thm:mr-corollary3}
Let $x_*$ be a fixed point of \eqref{eqn:dde}, and assume that $f$ is $C^2$ in a neighborhood of $x_*$.
If
\[
\alpha\big(M_0(x_*) + M_1(x_*) + \cdots + M_r(x_*)\big) < 0,
\]
then $x_*$ is locally exponentially stable for \eqref{eqn:dde} for any uniformly continuous time-delays $h(t)$.
\end{corollary}

\begin{proof}
    Without loss of generality suppose $x_*=0$.
    Label $f=f(t,x,y)$ where $y=(y_1,\ldots,y_r)\in \R^{rd}$.
    By standard constructions using bump functions, for any $r>0$ there is a $C^2$ function $g=g(t,x,y)$ such that the following hold:
    \begin{enumerate}[label=(\roman*)]
        \item inside $B_r(0)$, $g(t,x,y)=f(t,x,y)$;
        \item outside $B_{2r}(0)$, $g(t,x,y)=D_xf(t,0,0)x+D_yf(t,0,0)y$;
        \item for all $t,x,y$, $D_{(x,y)}g(t,x,y)=D_{(x,y)}f(t,x,y)+O(r^{-1})$ uniformly.
    \end{enumerate}
    Considering the DDE
    \begin{equation}\label{eqn:dde-local}
        \hat x'(t)=g(t,\hat x(t),\Delta_h \hat x(t)),
    \end{equation}
    note that $x_*=0$ is also a fixed point of \eqref{eqn:dde-local}.
    It is clear by property (iii) that the stability matrices of the system \eqref{eqn:dde-local} will be within $O(r^{-1})$ of the stability matrices of \eqref{eqn:dde}.
    Hence, for sufficiently small $r$, the system \eqref{eqn:dde-local} satisfies the conditions of Main Results 1, and $x_*$ is globally attracting for \eqref{eqn:dde-local}.
    In particular, there is a trapping neighborhood $U$ of $x_*$ contained within $B_r(0)$.
    Since within $B_r(0)$, the systems \eqref{eqn:dde} and \eqref{eqn:dde-local} are the same, any solution to \eqref{eqn:dde-local} with initial condition $\phi\in U$ will also be a solution to \eqref{eqn:dde}; hence $x_*$ is locally exponentially stable for \eqref{eqn:dde}.
\end{proof}

This proof also shows that the local exponential convergence rate can be computed analogously to Equation~\eqref{eqn:growth-rate2}.

\begin{corollary}[Local Convergence Rates]
Under the assumptions of \Cref{thm:mr-corollary3}, the exponential convergence rate $\gamma_0$ is given by \eqref{eqn:growth-rate2} with $M_i$ replaced by the local matrices $M_i(x_*)$. Explicitly,
\[
\gamma_0=\inf\Bigl\{\gamma\in \mathbb{R}:\,
\det\!\Bigl(\gamma I+M_0(x_*)+e^{\gamma T}\sum_{i=1}^r M_i(x_*)\Bigr)=0
\Bigr\}.
\]
Specifically, for the initial condition $\phi$ in a neighborhood of $x_*$, for any $0<\gamma<\gamma_0$,
\[
\|x(t)-x_*\|\leq C e^{-\gamma t}\sup_{s\leq 0}\|\phi(s)-x_*\|.
\]
\end{corollary}

\section{Application to Reservoir Computing}\label{sec:application}
We illustrate an application of intrinsic stability to reservoir computing, a class of machine learning methods for learning dynamical systems from time-series data.
Modern reservoir computers, developed from echo-state networks, have been successfully applied to attractor reconstruction \cite{chen2022}, forecasting \cite{pathak2018model}, classification \cite{bianchi2021}, and reconstruction from partial observations \cite{nakai2018}.

Continuous-time reservoir computers are typically constructed using a high-dimensional nonlinear ODE
\begin{equation}\label{eqn:reservoir-untrained}
x'(t)=F(x(t),u(t)),
\end{equation}
where $u(t)\in\mathbb{R}^m$ is an input signal on $I=[0,T]$. The state $x(t)\in\mathbb{R}^n$ is the reservoir response. From this response, a readout map $W:\mathbb{R}^n\to\mathbb{R}^m$ is trained so that $W(x(t))\approx v(t)$ for a target signal $v(t)$, typically via regularized least squares \cite{tavakoli2024delayedrc,chen2022}.

One property associated with high-performing reservoir computers is consistency, introduced in \cite{consistency2019}. 
Consistency quantifies the extent to which the reservoir response is determined by the input signal rather than by sensitivity to initial conditions or internal instability.

To define reservoir consistency, consider two solutions $x(t)$ and $y(t)$ of the same reservoir system driven by the same input $u(t)$ but with different initial conditions $x(0)\neq y(0)$. Over a time interval $[0,T]$, the consistency is measured by the mean Pearson correlation
\[
\gamma^2(T)=\frac{1}{nT}\sum_{i=1}^n\int_0^T\frac{(x_i(t)-\bar{x}_i(T))(y_i(t)-\bar{y}_i(T)}{\sigma_{i,x}(T)\sigma_{i,y}(T)}\,dt,
\]
where $\bar{x}_i(T),\bar{y}_i(T)$ denote time averages and $\sigma_{i,x}(T),\sigma_{i,y}(T)$ the corresponding standard deviations on the interval. The reservoir is said to be \textit{consistent} if $\gamma^2(T)\approx 1$ for all choices of initial conditions, meaning trajectories driven by the same input are highly correlated over $[0,T]$.

Recently, the authors of \cite{tavakoli2024delayedrc} showed that incorporating time delays into reservoir dynamics can improve performance, enabling effective classification of low-dimensional systems. 
They considered the delayed reservoir
\begin{equation}\label{eqn:reservoir-delayed-2}
\begin{aligned}
x_1'(t)=-x_1(t)-\delta x_2(t)+\frac{\beta}{M}&\sum_{i=1}^M\sin^2(x_1(t-\tau_i)+\phi+\gamma J(t))
\\
x_2'(t)&=x_1(t)
\end{aligned}
\end{equation}
where $x_1(t),x_2(t)\in\mathbb{R}$, $\beta,\delta,\phi,\gamma\in\mathbb{R}$ are parameters, $\tau_1,\ldots,\tau_M$ are fixed delays, and $J(t)$ encodes the input signal.

We show that, under mild assumptions on the input, the reservoir \eqref{eqn:reservoir-delayed-2} is consistent, even with time-varying delays. To establish this, we impose conditions on the input signal $J(t)$ and its effect on the system dynamics; these are typically satisfied in practice. In particular, $J(t)$ is usually sufficiently irregular to prevent convergence issues, and the reservoir admits no fixed points unless $J(t)$ is constant on an interval.

\begin{theorem}[Consistency of Delayed Reservoirs]\label{thm:consistency-delayed-2}
Let the delayed reservoir \eqref{eqn:reservoir-delayed-2} be driven by the input signal $J(t)$. Suppose there exists a bounded trajectory $x(t)$ for which the component means $\bar{x}_i(T)$ and 
standard deviations $\sigma_i(T)$ converge as $T\to\infty$ to $\bar{x}_i$ and $\sigma_i$, respectively, with $\sigma_i \neq 0$. If 
$$0<\beta<\tfrac{1}{2} \ \text{and} \ 0<\delta<\tfrac{1}{4}-\beta^2,$$ 
then the reservoir is consistent for any bounded uniformly continuous time delays $\tau_1,\ldots,\tau_M$.
\end{theorem}

\begin{proof}
We first show that \Cref{mr3} applies to the reservoir \eqref{eqn:reservoir-delayed-2}.
Write the right-hand-side of Equation \eqref{eqn:reservoir-delayed-2} in the notation \eqref{eqn:dde} as
\[
F(t,x,y_1,\ldots,y_k)=\pmat{
-x_1-\delta x_2+\frac{\beta}{k}\sum_{i=1}^k\sin^2(y_{i,1}+\phi+\gamma J(t))
\\
x_1
}.
\]
For notational simplicity, set $\Delta=\sqrt{1-4\delta}$, so $\delta=\frac14(1-\Delta^2)$.
Since $0<\delta<\frac14-\beta<\frac14$, we have $0<\Delta<1$; note also that $\delta<\frac14-\beta$ is equivalent to $2\beta<\Delta$.
We begin by making a change of variables.
Using the matrix 
\[
S=\pmat{\frac{-1+\Delta}2&\frac{-1-\Delta}2\\1 & 1},
\]
set $z(t)=S^{-1}x(t)$.
Then $z(t)$ satisfies the DDE 
\begin{equation}\label{eqn:reservoir-cov}
\begin{aligned}
z(t)=G(t,z(t),&z(t-\tau_1),\ldots,z(t-\tau_k)),
\\
G(t,z,w_1,\ldots,w_k)=&S^{-1}F(t,Sz,Sw_1,\ldots,Sw_k).
\end{aligned}
\end{equation}
The matrix $S$ is chosen to diagonalize $D_x F$.
We compute the intrinsic stability matrix for Equation \eqref{eqn:reservoir-cov}; as $S$ induces a bi-uniformly continuous conjugacy between the systems \eqref{eqn:reservoir-delayed-2} and \eqref{eqn:reservoir-cov}, the results of \Cref{mr3} can be passed between the two systems.
One may then verify that
\[
D_zG=SD_xFS^{-1}
=
\pmat{
\frac{-1-\Delta}2 & 0 \\ 0 & \frac{-1+\Delta}2
},
\]
\[
D_{w_i}G=SD_{y_i}FS^{-1}
=
\frac{\beta}{2k\Delta}\pmat{1+\Delta & 1-\Delta \\ -1-\Delta & -1+\Delta}\sin(2w_{i,1}+2\phi+2\gamma J(t)).
\]
Hence 
\[
M_0=\pmat{\frac{-1-\Delta}2&0\\0&\frac{-1+\Delta}2},
\quad\quad
M_i
= \frac{\beta}{2k\Delta}\pmat{1+\Delta & 1-\Delta \\ 1+\Delta & 1-\Delta}
\]
for $1\leq i \leq k$.
The sum of the stability matrices is thus given by
\[
\mathcal M=M_0+M_1+\ldots + M_k
=
\pmat{
\frac{-1-\Delta}2 & 0 \\ 0 & \frac{-1+\Delta}2
}
+
\frac{\beta}{2\Delta}\pmat{1+\Delta & 1-\Delta \\ 1+\Delta & 1-\Delta}.
\]
For parameters in the range considered, the eigenvalues of this matrix are both real, and the larger is
\[
\alpha(\mathcal M)
=
\frac12\paren{-1+\frac\beta{\Delta}+\sqrt{
\frac{\beta^2}{\Delta^2}-2\beta\Delta+\Delta^2
}}.
\]
It follows that for $0<\Delta\leq 1$, $\alpha(\mathcal M)=0$ if $\Delta=2\beta$ or $\Delta=1$, $\alpha(\mathcal M)>0$ for $\Delta\in(0,2\beta)$, and $\alpha(\mathcal M)<0$ for $\Delta\in(2\beta,1)$.
Hence, under our assumptions on $\Delta$ and $\beta$, $\alpha(\mathcal M)<0$.
Thus, by \Cref{mr3}, the reservoir \eqref{eqn:reservoir-cov} has exponential independence to initial conditions, which as argued above also applies to the reservoir  \eqref{eqn:reservoir-delayed-2}.

Now let $y(t)=(y_1(t),y_2(t)),z(t)=(z_1(t),z_2(t))$ be two responses to the same input signal $u(t)$ as $x(t)$.
By exponential convergence of trajectories it follows that $x(t)-y(t)=O(e^{-\gamma_0 t})$ and $x(t)-z(t)=O(e^{-\gamma_0 t})$.
In particular, this implies $\bar y_i(T),\bar z_i(T)\to \mu_i$ and $\sigma_{i,y}(T),\sigma_{i,z}(T)\to \sigma_i$.
By our assumptions on the trajectory, as $T\to\infty$,
\[
\gamma^2(T)
=
\frac{1}{nT}\sum_{i=1}^n\int_0^T\frac{(x_i(t)-\bar{x}_i(T))(y_i(t)-\bar{y}_i(T))}{\sigma_{i,x}(T)\sigma_{i,y}(T)}\,dt
\]\[
=
\frac{1}{nT}\sum_{i=1}^n\frac1{\sigma_i^2+O(1)}\int_0^T\paren{\paren{x_i(t)-\mu_i}^2+O(1)}
\,dt
\to
1
\]
and the reservoir is consistent.
\end{proof}

Figure~\ref{fig:delayed-reservoir} demonstrates consistency of the delayed reservoir \eqref{eqn:reservoir-delayed-2}. The input $J(t)$ is taken from the $x$-coordinate of the Lorenz attractor over $[0,30]$, with constant delays $\tau_1=0.4$, $\tau_2=0.7$, and $\tau_3=1.0$. 
We note that intrinsic stability can be used to derive consistency criteria for other reservoir architectures.

\begin{figure}
\center
\includegraphics[width=4.5in]{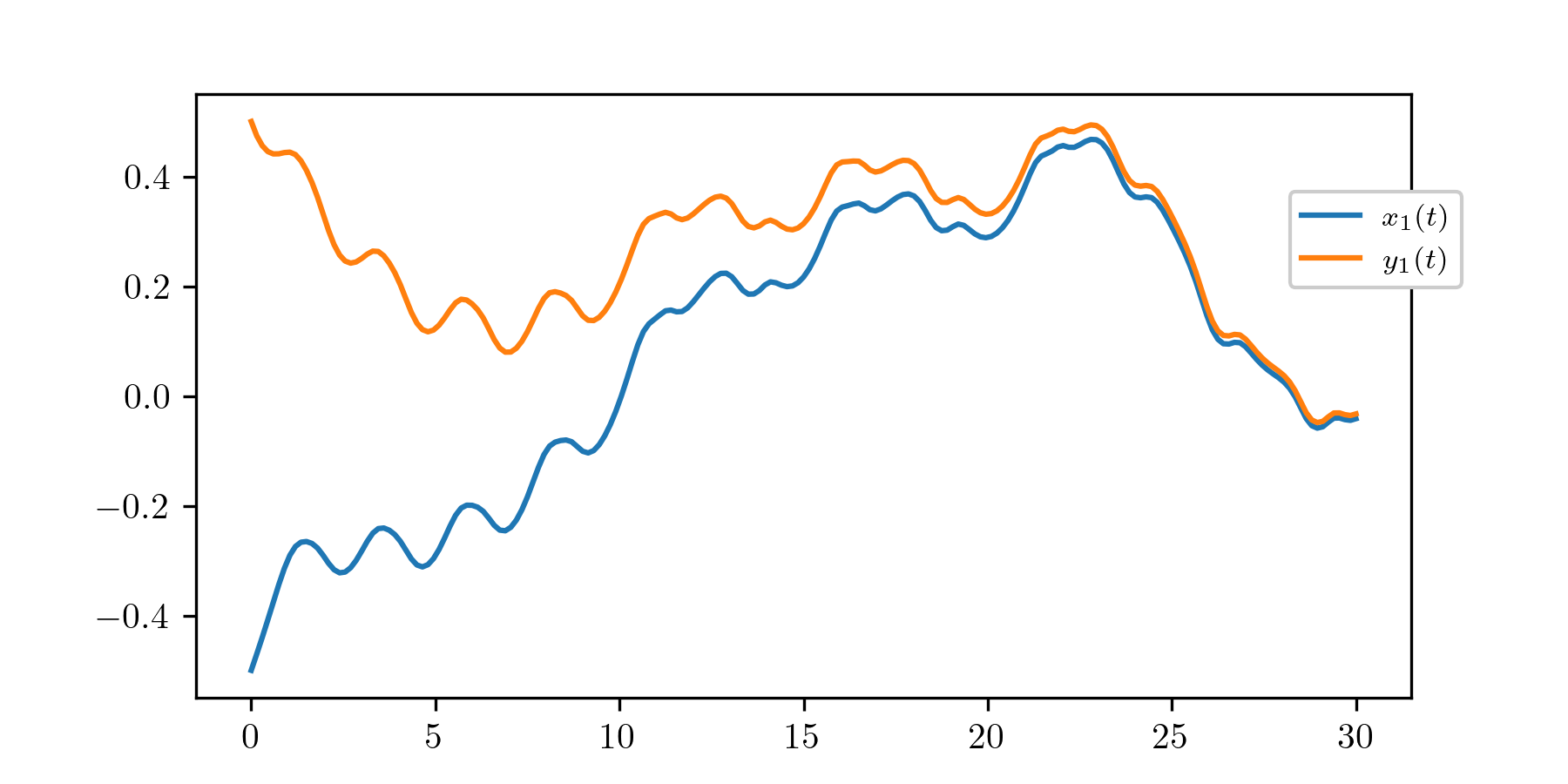}
\caption[Example of delayed reservoir computer consistency]{
Consistency of the delayed reservoir \eqref{eqn:reservoir-delayed-2}. The curves show the first components $x_1(t)$ and $y_1(t)$ of two trajectories for different initial conditions, which converge over time. The parameters $\beta=1/3$ and $\delta=1/8$ satisfy \Cref{thm:consistency-delayed-2} \textcolor{black}{guaranteeing consistency. Here $k=3$ with delays} $\tau_1=0.4$, $\tau_2=0.7$, $\tau_3=1.0$, for $\phi=1$ and $\gamma=7$. The input $J(t)$ is derived from the Lorenz attractor \cite{Lorenz1962}.
}\label{fig:delayed-reservoir}
\end{figure}

\section{Conclusion}\label{sec:conclusion}

This work extends the theory of intrinsic stability from discrete-time systems to a broad class of continuous-time, nonlinear, nonautonomous delay differential equations. We established a delay-independent criterion for exponential stability based on the spectral properties of associated stability matrices, together with explicit bounds on convergence rates. To our knowledge, this is the first general result of this type for nonlinear DDEs.

A key feature of the proposed framework is its computational simplicity: in contrast to Lyapunov-based or LMI methods, the stability criterion depends only on matrix spectra, which are typically inexpensive to evaluate. This makes the approach both theoretically transparent and accessible.

Beyond global stability, we demonstrate how this framework yields other results, including the existence of globally attracting equilibria and periodic solutions, as well as delay-independent local stability conditions. As an application, we analyzed consistency in delayed reservoir computing systems, illustrating how intrinsic stability provides a natural tool for assessing input-driven dynamics in machine learning models.

Several directions for future work remain. First, it would be of interest to extend the theory to more general classes of delays, including distributed, state-dependent, and stochastic delays. Second, the relationship between our stability condition and more general switching criteria (see e.g., \cite{sun2012delayindep}) suggests the possibility of extending the present results to broader classes of nonlinear systems. Finally, it remains an open question whether the intrinsic stability condition is necessary for delay-independent stability, modulo the coordinate dependence of the stability matrices (compare, for instance, the proof of Theorem \ref{thm:consistency-delayed-2}). Understanding this issue may shed light on analogous open problems in the discrete-time setting \cite{Reber_2020}.

\section*{Funding}
B.Z.W. was partially supported by the NSF grant \#2205837 in Applied Mathematics.

\bibliographystyle{elsarticle-num} 
\bibliography{main}

\setcounter{section}{0}
\renewcommand{\thesection}{Appendix \Alph{section}}
\section{Additional proofs}
\renewcommand{\thesection}{\Alph{section}}
In this appendix are recorded proofs of several propositions and lemmas with long proofs whose ideas are not central to the sections they appear in.

\begin{proof}[Proof of \Cref{thm:continuity}]
    Each of the items (i)-(iv) follows by a direct application of Gronwall's inequality, in a manner very similar to the same statements for ODEs.
    We show the proof only of (iv), as it is the only statement without an ODE analogue. 
    Let $y_1(t)$ and $y_2(t)$ be the solutions to Equation \eqref{eqn:dde} with initial condition $\phi$ and time delays $g$ and $h$ respectively; then 
    \[
    y_1(t)=\begin{cases}
        S^g_t[\phi](0) & t\geq 0
        \\
        \phi(t) & t\leq 0
    \end{cases}
    \]
    and similar for $y_2$.
    Using Lipschitz continuity of $f$ (including in the variable $t$) it is straightforwards to show via Gronwall's inequality that
    \begin{equation}\label{eqn:cont-part1}
     \norm{S^h_t[\phi]}_{C^0}\leq Ce^{Lt}(1+\norm{\phi}_{C^0}).
    \end{equation}
    Using \eqref{eqn:cont-part1} with Lipschitz continuity and relabeling constants as needed, we may then estimate 
    \begin{align*}
    \Lip(S^h_t[\phi])
    &\leq \Lip(\phi)+\sup_{0\leq s\leq t}\norm{y_2'(s)}
    \\&= \Lip(\phi)+\sup_{0\leq s\leq t}\norm{f(s,y_2(s),\Delta_h y_2(s))}
    \\&\leq \Lip(\phi)+Lt+L(r+1)\norm{S^h_t[\phi]}_{C^0}
    \\&\leq Ce^{L't}(1+\norm{\phi}_{C^0,1}).\eqlabel{eqn:cont-part2}
    \end{align*}
    We may apply \eqref{eqn:cont-part2} to then show
    \begin{align*}
        \AnchorRight\norm{f(t,y_1(t),\Delta_gy_1(t))-f(t,y_2(t),\Delta_hy_2(t))}
        \\&\leq
        L\norm{(y_1(t),\Delta_g y_1(t))-(y_2(t),\Delta_g y_2(t))}
        +L\norm{\Delta_gy_2(t)-\Delta_h y_2(t)}
        \\&\leq
        L(r+1)\norm{S^g_t[\phi]-S^h_t[\phi]}_{C^0}+L\Lip(S^h_t[\phi])\norm{g(t)-h(t)}
        \\&\leq
        L'\norm{S^g_t[\phi]-S^h_t[\phi]}_{C^0}+Ce^{L't}(1+\norm{\phi}_{C^{0,1}})\norm{g(t)-h(t)}.\eqlabel{eqn:cont-part3}
    \end{align*}
    This allows us to then estimate
    \begin{align*}
        \norm{S^g_t[\phi]-S^h_t[\phi]}_{C^0}&=\sup_{t-T\leq s \leq t}\norm{y_1(s)-y_2(s)}
        \\&\leq \int_0^t \norm{f(s,y_1(s),\Delta_gy_1(s))-f(t,y_2(s),\Delta_hy_2(s))}\,ds
        \\&\leq
        L' \int_0^t\norm{S^g_s[\phi]-S^h_s[\phi]}_{C^0}
        +\norm{g-h}_{L^\infty(0,t)}\frac C{L'}(1+\norm{\phi}_{C^0,1}) e^{L't}\eqlabel{eqn:cont-part4}
    \end{align*}
    whence Gronwall's inequality implies
    \begin{equation}\label{eqn:cont-part5}
    \norm{S^g_t[\phi]-S^h_t[\phi]}_{C^0}
    \leq
    C'\norm{g-h}_{L^\infty(0,t)}(1+\norm{\phi}_{C^0,1})e^{2L't}.
    \end{equation}
    To complete statement (iv), it then suffices to additionally consider
    \begin{align*}
    \Lip(S^g_t[\phi]-S^h_t[\phi])&=\sup_{t-T\leq s \leq t}\norm{y_1'(s)-y_2'(s)}
    \\&=\sup_{t-T\leq s \leq t}\norm{f(s,y_1(s),\Delta_gy_1(s))-f(t,y_2(s),\Delta_hy_2(s))}
    \end{align*}
    and apply \eqref{eqn:cont-part3} and \eqref{eqn:cont-part5}.
\end{proof}

\begin{proof}[Proof of \Cref{thm:discretization-spectrum}] Write $R=R_1\cdots R_m$ and $A=A_1\cdots A_m$.
Since $\pi_\tau\circ R_i=A_i\circ\pi_\tau$ for each $1\leq i \leq m$, it follows inductively that $\pi_\tau \circ R=A\circ \pi_\tau$.
Choose $\ell\geq 1$ large enough that $\ell m\tau \geq T$.
Noting that $R^\ell\in \Sigma_\tau(\ell m\tau)$,
\Cref{thm:discretization-kernel} implies that
\[
\ker(\pi_\tau)\subseteq \ker (R^\ell).
\]
Since $\ker(\pi_\tau)$ has finite codimension $n_\tau$, the range of $R^\ell$ must be finite-dimensional and so $R^\ell$ is compact.
By \cite[Chapter VII.4, Theorem 6]{dunford1988linearpart1},
all nonzero elements of the spectrum of $R$ are eigenvalues, or elements of the point spectrum $\sigma_p(R)$, so
$\sigma(R)=\sigma_p(R)\cup\set{0}$.
To prove the claim, it suffices to show that
\[
\sigma_p(R)\cup\set{0}
=
\sigma(A)\cup\set{0}.
\]
First, let $\lambda\in \sigma_p(R)$ be an arbitrary eigenvalue with corresponding eigenfunction $x$, so $Rx=\lambda x$.
If $x\in \ker(\pi)$, we have that $x\in \ker(R^\ell)$ also, so 
\[
0
=
R^\ell x
=
\lambda^\ell x,
\]
implying that $\lambda=0$.
Otherwise, $\pi_\tau x\neq 0$, and we have
\[
A\pi_\tau x
=
\pi_\tau Rx
=
\pi_\tau \lambda x
=
\lambda \pi_\tau x,
\]
so $\pi_\tau x$ is an eigenvector of $A$ with eigenvalue $\lambda$.
Thus $\sigma_p(R)\cup\set{0}\subseteq\sigma(A)\cup\set{0}$.

To show the reverse inclusion, let $\lambda\in \sigma(A)$ be nonzero with corresponding eigenvector $v$.
Since $\pi_\tau$ is surjective, there exists some $x$ with $\pi_\tau x=v$.
We claim $R^\ell x$ is an eigenfunction of $R$ with eigenvalue $\lambda$.
To see this, observe that
\[
\pi_\tau(Rx-\lambda x)
=
A\pi_\tau x-\lambda \pi_\tau x
=
Av-\lambda v
=
0,
\]
so $Rx-\lambda x\in \ker(\pi_\tau)\subseteq \ker(R^\ell)$.
Thus,
\[
R(R^\ell x)-\lambda(R^\ell x)
=
R^\ell (Rx-\lambda x)=0.
\]
Moreover,
\[
\pi R^\ell x
=
A^\ell \pi x
=
A^\ell v
=
\lambda^\ell v\neq 0,
\]
so $R^\ell x\neq 0$, and $R^\ell x$ is an eigenfunction of $R$ with eigenvalue $\lambda$.
Hence $\sigma_p(R)\cup\set{0}
=
\sigma(A)\cup\set{0}$, completing the proof.
\end{proof}

\begin{proof}[Proof of \Cref{thm:splitting-value}]
Suppose that $h\in\overline{\mathrm{LI}}$.
Then, there is a sequence $\hat h^n\in \mathrm{LI}_{\tau_n}$ with $\hat h^n\to h$ for some values $\tau_n>0$.
Since $\mathrm{LI}_{t}\subseteq \mathrm{LI}_{t/m}$ for any $m\geq 1$, we may assume without loss of generality assume that $\tau_n\to 0$ monotonically.
If some infinite subsequence of $\set{\tau_n}$ are co-rational (i.e. $\tau_{n_k}=q_k\tau_{n_{k+1}}$ for some $q_k\in \Q$), then by further sub-dividing the values of $\tau_{n_k}$, we can require $\tau_{n_{k+1}}\mid \tau_{n_k}$, which forces each $\tau_{n_k}$ to be of the form $\tau_{n_0}/m_k$ for some integers $m_k$.

Hence it suffices to consider the case where all but finitely many of the $\tau_n$ are co-irrational.
We claim that this implies $h$ must be uniformly continuous, in which case $\hat h^n$ satisfying (i) may be constructed as in \Cref{thm:uc-satisfies-assumption}.
Fix any $\epsilon>0$, and choose $\hat h^1,\hat h^2$ with $\normsize{}{\hat h^i-h}<\epsilon$ such that $\tau_1$ and $\tau_2$ are co-irrational.
First, note that if $t_1,t_2\in (k \tau_i,(k+1)\tau_i)$, then
\begin{equation}\label{eqn:splitting-1}
\norm{h(t_1)-h(t_2)}\leq
\normsize{}{h(t_1)-\hat h^i(t_1)}
+
\normsize{}{\hat h^i(t_1)-\hat h^i(t_2)}
+
\normsize{}{\hat h^i(t_2)- h(t_2)}
\leq 2\epsilon +\abs{t_1-t_2}.
\end{equation}
Now, fix $\delta < \min(\tau_1/2,\tau_2/2,\epsilon)$.
Then, for any $t_1\in [0,\infty)$, the interval $(t_1-\delta,t_1+\delta)$ can be covered by at most four open intervals of the form $(n\tau_i,(n+1)\tau_i)$ for $n\in\N$ (note that $\tau_i \N\cap (t_1-\delta,t_1+\delta)$ can contain at most one point due to our choice of $\delta$).
Hence using Equation \eqref{eqn:splitting-1} within each of the intervals $(n\tau_i,(n+1)\tau_i)$, the triangle inequality implies that for all $t_2\in (t_1-\delta,t_1+\delta)$,
\[
\norm{h(t_1)-h(t_2)}\leq 8\epsilon+\abs{t_1-t_2}\leq 9\epsilon
\]
which shows uniform continuity.
This establishes (i).

Property (ii) follows quickly from (i) using the characterization of precompact subsets of $L^\infty([0,t])$ given in \cite[Chapter IV.8, Theorem 18]{dunford1988linearpart1}.
Namely, for some $n>0$, partition $[0,t]$ into $I_1=[0,t/n)$, $I_2=[t/n,2t/n)$, and so forth, and 
let $U[g]:[0,t]\to \R^d$ be given by
\[
U[g](s)=\frac nt\int_{(j-1)t/n}^{jt/n} g(s')\,ds'\quad\quad \text{if $s\in [(j-1)t/n,jt/n)$}.
\]
To show precompactness it suffices to show that $\normsize{}{U[h^{k;t}]-h^{k;t}}_{L^\infty([0,t])}$ may be made arbitrarily small, uniformly in $k$.
Fix $\epsilon>0$ and choose $n>1/\epsilon$ such that $\normsize{}{\hat h^n-h}< \epsilon$.
If $t_1,t_2\in [(k-1)t,kt)$ for any $k\geq 1$, we have that 
\[
\norm{h(t_1)-h(t_2)}_\infty\leq \normsize{}{h(t_1)-\hat h^n(t_1)}+\normsize{}{\hat h^n(t_1)-\hat h^n(t_2)}+\normsize{}{\hat h(t_2)-h(t_2)}<2\epsilon+\abs{t_1-t_2}.
\]
Consequentially, as $h^{k;t}(s)=h(s+kt)$, it is the case that $\sup_{s\in I_j}h^{k;t}(s)-\inf_{s\in I_j}h^{k;t}(s)\leq 2\epsilon+\abs{I_j}\leq 3\epsilon$ for each $1\leq j \leq n$ and each $k\geq 0$.
This then implies that 
\[
-3\epsilon\leq 
U[h^{k;t}](s)-h^{k;t}(s)\leq 3\epsilon
\]
for all $s\in [0,t]$ and all $k\geq 0$, which implies the set is precompact by the above-cited result as desired.
\end{proof}

\end{document}